\documentclass{amsart}
\usepackage{amsmath}
\usepackage{amsxtra}
\usepackage{amstext}
\usepackage{amssymb}
\usepackage{latexsym}
\usepackage{dsfont} 
\usepackage{verbatim}
\usepackage{tabls}
\usepackage{rotating}

\newtheorem{thm}{Theorem}[section]
\newtheorem{cor}[thm]{Corollary}
\newtheorem{lem}[thm]{Lemma}
\newtheorem{prop}[thm]{Proposition}
\newtheorem{defn}[thm]{Definition}

\newtheorem{rem}[thm]{Remark}
\newcommand{\abs}[1]{\left\vert#1\right\vert}
\newcommand{\norm}[1]{\left\Vert#1\right\Vert}

\numberwithin{equation}{section}

\def\Xint#1{\mathchoice
   {\XXint\displaystyle\textstyle{#1}}%
   {\XXint\textstyle\scriptstyle{#1}}%
   {\XXint\scriptstyle\scriptscriptstyle{#1}}%
   {\XXint\scriptscriptstyle\scriptscriptstyle{#1}}%
   \!\int}
\def\XXint#1#2#3{{\setbox0=\hbox{$#1{#2#3}{\int}$}
     \vcenter{\hbox{$#2#3$}}\kern-.5\wd0}}

\def\dashint{\Xint-}

\begin{document}
\title[Nonlinear equations with natural growth terms]{The fundamental solution of nonlinear equations with natural growth terms}

\author{Benjamin J. Jaye}
\address{Department of Mathematics, University of Missouri, Columbia, Missouri 65211, USA}
\email{bjjm93@mizzou.edu}

\author{Igor E. Verbitsky}
\address{Department of Mathematics, University of Missouri, Columbia, Missouri 65211, USA}
\email{verbitskyi@missouri.edu}

\thanks{Supported in part
by NSF grant  DMS-0901550.}




\begin{abstract}
We find bilateral global bounds for the fundamental solutions associated with some quasilinear and fully nonlinear operators perturbed by a nonnegative zero order term with natural growth under minimal assumptions.   Important  model problems 
involve the equations $-\Delta_p u = \sigma \abs{u}^{p-2}u + \delta_{x_0}$, for $p>1$, and $F_k(-u) = \sigma \abs{u}^{k-1}u + \delta_{x_0}$, for $k \geq 1$.  Here $\Delta_p$ and  $F_k$ are the $p$-Laplace and $k$-Hessian operators respectively, and $\sigma$ is an 
arbitrary positive measurable function (or measure).  We will in addition consider the Sobolev regularity of the fundamental solution away from its pole.
\end{abstract}
\maketitle

\section{Introduction}\label{intro}
\subsection{} In this paper we study the fundamental solution associated with certain nonlinear operators perturbed by natural growth terms.  Consider, for $1<p<\infty$, the quasilinear operator 
\begin{equation}\label{defnl1}
\mathcal{L}(u) = \mathcal{L}^{(p)}(u) = -\Delta_p u - \sigma \abs{u}^{p-2}u,
\end{equation}
where  $\Delta_p u= \text{div} (\nabla u \abs{\nabla u}^{p-2})$ is the $p$-Laplacian operator and $\sigma$ is a nonnegative Borel measure, on $\mathds{R}^n$.

Our main goal is to investigate the interaction between the differential operator $-\Delta_p u$, and the lower order term $\sigma \abs{u}^{p-2}u$, under necessary conditions on $\sigma$.  This interaction between the differential operator and the lower order term turns out to be highly nontrivial.   We will also study the corresponding problem when the $p$-Laplacian is replaced by a more general quasilinear operator, or a fully nonlinear operator of Hessian type. 

Our theorems  extend to nonlinear operators very recent results \cite{FV1}, \cite{FNV}, 
\cite{GriHan1} regarding the behavior of the Green function of the time independent Schr\"{o}dinger operator 
 $-\Delta u - \sigma u$. Our approach, which combines some nonhomogeneous harmonic analysis, nonlinear 
potential theory and PDE methods, is based on a certain discrete ``pseudo-probabilistic'' model of equation (\ref{defnl1}), which employs a family of nonlinear expectation operators 
(see Section~\ref{reduction} below).

This method allows us to construct fundamental solutions of the operator $\mathcal{L}$ under assumptions on $\sigma$ so in general the Harnack inequality fails for positive solutions $u$ of $\mathcal{L}(u)=0$.  The Harnack inequality formed an essential part in classical arguments concerning the construction of fundamental solutions to both linear and nonlinear operators \cite{LSW63, Roy62, Ser1, Ser2}.  For example, our results hold for the Hardy potential $\sigma(x) = c|x|^{-p}$ for $0<c<((n-p)/p)^p$.

Now consider the equation:
\begin{equation}\label{fundsolnintro}
\mathcal{L}(u) = \delta_{x_0} \quad \text{in} \;\;\mathds{R}^n, \qquad \inf_{x\in \mathds{R}^n} u(x) = 0 ,
\end{equation}
where $\delta_{x_0}$ is the Dirac delta measure concentrated at $x_0$.  A solution $u(x,x_0)$ of (\ref{fundsolnintro}) understood in a suitable weak, or potential theoretic sense 
(e.g. renormalized, viscosity, or approximate solutions), is called a \textit{fundamental solution} of the operator $\mathcal{L}$, with pole at $x_0$.

It is well known (\cite{Ser2}, \cite{Ser1}, \cite{veron}) that, under stringent assumptions on $\sigma$, there exists a 
positive  constant $c$ so that 
 \begin{equation}\label{fundsolnequivineq} 
\frac{1}{c} \,  G(x,x_0)\le u(x,x_0) \le c \,  G(x,x_0),
\end{equation}
if $\abs{x-x_0} <R$ for some $R>0$, where $G(x,x_0)$ is the fundamental solution  of 
$\Delta_p$  
on $\mathds{R}^n$: 
\begin{equation}\label{unperturbed}
G(x,x_0) = \gamma_{p,n}\abs{\,x-x_0}^{\frac{p-n}{p-1}}, 
\quad \text{ when } 1<p<n.
\end{equation}  
 Here $\gamma_{p,n} =  \frac{p-1}{n-p} \bigl(n \omega_{n-1}\bigl)^{-\frac 1{p-1}}$ and $\omega_{n-1}$ is the surface area of the $n-1$ dimensional sphere in $\mathds{R}^n$. Moreover, 
 it was shown recently by L. Ver\'{o}n (see \cite{PT1}, Lemma 5.1) that $\lim_{x \to x_0} u(x,x_0)/G(x,x_0)=c$ if $\sigma \in L^{\infty}_{loc} 
(\mathds{R}^n)$. However, as we will  see below, $u(x,x_0)$  may behave very differently in comparison to $G(x, x_0)$, both locally and 
globally. 

In this paper we will obtain sharp global estimates for the behavior of fundamental solutions: 
  \textit{Suppose $1<p<n$. Then any fundamental solution $u(x,x_0)$ with pole at $x_0$ satisfies the following lower bound:}
\begin{equation}\begin{split}\label{l1lowest} u(x,x_0) \geq  c\abs{x-x_0}^{\frac{p-n}{p-1}}& \exp\Big( c\int_{0}^{\abs{x-x_0}}\Bigl(\frac{\sigma(B(x,r)}{r^{n-p}}\Bigl)^{1/(p-1)}\frac{dr}{r}\Bigl)\\
&\cdot \exp\Bigl(c\int_{0}^{\abs{x-x_0}} \frac{\sigma(B(x_0,r))}{r^{n-p}}\frac{dr}{r}\Bigl),
\end{split}\end{equation}
for any $x, x_0 \in \mathds{R}^n$ under necessary conditions on the measure $\sigma$.  Here $c$ is a positive constant depending on $n$ and $p$, and $B(x,r)$ is a ball of radius $r$ centered at $x$. 

 The sharpness of this lower bound is illustrated explicitly by our primary result:    \textit{Under a natural assumption on $\sigma$,  there exists a fundamental solution $u(x,x_0)$ of $\mathcal{L}$ satisfying the corresponding upper bound, i.e. for another positive constant $c$, depending on $n, p$ and $\sigma$, it holds that:}
\begin{equation}\begin{split}\label{l1upest}
u(x,x_0) \leq  c\abs{x-x_0}^{\frac{p-n}{p-1}}& \exp\Big( c\int_{0}^{\abs{x-x_0}}\Bigl(\frac{\sigma(B(x,r))}{r^{n-p}}\Bigl)^{1/(p-1)}\frac{dr}{r}\Bigl)\\
&\cdot \exp\Bigl(c\int_{0}^{\abs{x-x_0}} \frac{\sigma(B(x_0,r))}{r^{n-p}}\frac{dr}{r}\Bigl).
\end{split}\end{equation}
 See Theorems \ref{l1lowbd} and \ref{l1upbd} below for more precise statements.  Furthermore, it follows that there is a minimal fundamental solution which obeys (\ref{l1lowest}) and (\ref{l1upest});  see Corollary \ref{minbnds}.  These results had previously been announced without proofs in \cite{verb}. 
 
 In addition to the pointwise bounds presented above, the regularity of the constructed fundamental solution $u(x,x_0)$ away from the pole $x_0$ will be considered.  In particular it will be proved that $u( \cdot, x_0) \in W^{1,p}_{\text{loc}}(\mathds{R}^n\backslash\{x_0\})$, see Theorem \ref{fundreg}.  This is the optimal regularity that one can hope for under our assumption on $\sigma$, see Remark \ref{fundrem} below. 

\begin{rem} It is somewhat surprising that expressions involving both the linear potential
$  \displaystyle \mathbf{I}^{\rho}_p \sigma (x_0) = \int_{0}^{\rho} \frac{\sigma(B(x_0,r))}{r^{n-p}}\frac{dr}{r}$ of fractional order $p$, and 
the nonlinear Wolff's potential, introduced in \cite{HW},  $$\mathbf{W}_{1,p}^{\rho} \sigma  (x)= \int_{0}^{\rho}\left (\frac{\sigma(B(x,r))}{r^{n-p}}\right)^{1/(p-1)}\frac{dr}{r},$$ 
should appear, in the exponential form, in  global bounds of  solutions  
of the equation $-\Delta_p u - \sigma \abs{u}^{p-2}u = \delta_{x_0}$. 
\end{rem} 
We observe that local Wolff's potential estimates of solutions of the equation $-\Delta_p u = \sigma$ were established by Kilpel\"{a}inen and Maly in \cite{KM1}, while the fully nonlinear analogues for Hessian equations are due to Labutin  \cite{Lab1}.

  A simple corollary of our results (Corollary \ref{fundsolnequiv} below) gives necessary and sufficient conditions on $\sigma$ which ensure that $u(x,x_0)$  and $G(x,x_0)$ are pointwise comparable globally.  
  This requires the uniform boundedness of the Riesz potential 
  $\mathbf{I}_p \sigma$ when $1<p\le 2$ and the Wolff potential 
  $\mathbf{W}_{1,p} \sigma$ when $p>2$:  
  
  \textit{ Suppose there is a constant $c>0$ so that  (\ref{fundsolnequivineq}) holds for all $x, x_0\in  \mathds{R}^n$. 
  Then necessarily,}
\begin{align}\label{item1fund}
 \displaystyle  \sup_{x\in \mathds{R}^n} &  \int_{0}^{\infty} \frac{\sigma(B(x,r))}{r^{n-p}}\frac{dr}{r} < \infty  \quad  \rm{if} \quad 1<p \leq 2,  \\
\label{item2fund} \displaystyle  \sup_{x\in \mathds{R}^n} &  \int_{0}^{\infty}\left (\frac{\sigma(B(x,r))}{r^{n-p}}\right)^{1/(p-1)}\frac{dr}{r} < \infty \quad  \rm{if} \quad p > 2.
\end{align}
\textit{Conversely, 
 (\ref{item1fund})--(\ref{item2fund}) are sufficient for (\ref{fundsolnequivineq})  to hold for all $x, x_0\in  \mathds{R}^n$,  under a natural smallness assumption on $\sigma$ discussed below.}

 In a recent paper of Liskevich and Skrypnik \cite{LS1}, an indication 
 of this behavior involving the linear potential $\mathbf{I}_p(\sigma)$ when $1<p\le 2$ appeared for the first time. They studied isolated singularities of  operators of the type $\mathcal{L}u = -\Delta_p u - \sigma \abs{u}^{p-2}u$, under the assumption that  $\sigma$ is in the quasilinear Kato class (see, 
 e.g.,  \cite{Bir01}):
\begin{equation}\label{kato}
\lim_{\rho\to 0^+} \sup_{x \in  \mathds{R}^n} \int_{0}^{\rho} \left ( \frac{| \sigma | (B(x,r))}{r^{n-p}}\right )^{1/(p-1)}\frac{dr}{r}  =0. 
\end{equation}

In this paper we will assume that  $\sigma$ 
is a positive Borel measure satisfying the following capacity condition:
\begin{equation}\label{capintro}
\sigma(E) \leq C \, \rm{cap}_{p}(E) \;\text{ for any compact set } E\subset \mathds{R}^n,
\end{equation}
where $\text{cap}_p$ is the standard $p$-capacity:
\begin{equation}\label{pcapintro}
\text{cap}_{p}(E) = \inf\{ \;\norm{\nabla f}^p_{\text{L}^p} \;: \; f\geq 1 \; \text{on} \; E, \, \, f \in C^\infty_0(\mathds{R}^n)\;\}.\end{equation}
This intrinsic condition, which originated in the work of Maz'ya  in the 
context of linear problems  (see \cite{MazSob}), is  less stringent than the quasilinear  Kato condition (\ref{kato}).  However, when working in this generality, we cannot expect solutions to be continuous or satisfy a Harnack inequality.

It is easy to see that (\ref{capintro}) with constant $C=1$ is necessary 
in order that $u(\cdot, x_0)$ be finite a.e., which  is an immediate 
consequence of  the 
inequality
\begin{equation}\label{supersol}
\int_{\mathds{R}^n}  |h|^p \, d \sigma \le \int_{\mathds{R}^n} |\nabla h|^p \, dx, \qquad h \in C^\infty_0(\mathds{R}^n). 
\end{equation}
The preceding inequality holds whenever there exists a positive supersolution $u$ 
so that $-\Delta_p u \ge \sigma u^{p-1}$ (see Section~\ref{reduction}). We observe that, in its turn, 
 (\ref{capintro}) with $C = (p-1)^p/p^p$ yields (\ref{supersol}) 
 (see  \cite{MazSob}). 

\subsection{}  Recall that the fundamental solution of the Laplacian operator plays an important role in the theory of harmonic functions not only because of the principle of superposition, but also because of its importance in understanding how solutions near an isolated singularity can behave, see e.g. Theorem 1.3.7 of \cite{AG01}.  The latter theory carries over to the theory of the quasilinear and fully nonlinear operators considered here, and hence from the bounds for the fundamental solution we deduce a rather complete analysis of the behavior of solutions of $\mathcal{L}(u) = 0$, and the analogue for the $k$-Hessian operator, in the punctured space.  For the quasilinear operator, this has been considered under a variety of assumptions on $\sigma$ in \cite{LS1, NSS, Ser1, Ser2, veron}.  Isolated singularities of nonlinear operators have been studied recently in \cite{Labiso, Yanyaniso}.  We will present this application in a forthcoming note, where we will also consider other applications, for instance to the study of sign changing solutions of the equation:
 \begin{equation}\label{order1}-\Delta_p u = \abs{\nabla u}^p + \sigma, \end{equation}
see, for instance \cite{KK1, FM1, HMV, AHBV, MP1} for some of the existing literature regarding (\ref{order1}).

 \subsection{} The plan of the paper is as follows.  In 
 Section~\ref{main} we precisely state our main results regarding the fundamental solution of (\ref{defnl1}) and its fully nonlinear analogue. 
 
 In Section \ref{background}, we rapidly review some elements of the theory of nonlinear PDE from a potential theoretic perspective.  We are essentially interested in two aspects of this theory: potential estimates for solutions, and weak continuity of the elliptic operators.  In this section we also collect a few facts about capacities, and discuss minimal fundamental solutions.  After this, in Section \ref{reduction}, we discuss how the potential estimates reduce matters to the study of certain nonlinear integral inequalities.  In this section we also discuss the necessary capacity conditions on the measure $\sigma$ in order for positive solutions of the differential inequalities $\mathcal{L}u\geq 0$ or $\mathcal{G} u \geq 0$ to exist.

Section \ref{lowbds} is concerned with finding a lower bound for any positive solution of a certain nonlinear integral inequality.  This bound is proved by estimating successive iterations of the inequality by induction.  From this bound Theorems \ref{l1lowbd} and \ref{l2lowbd} are deduced, and their proofs conclude Section \ref{lowbds}.

In Section \ref{construct}, we consider the problem of constructing a positive solution to the integral inequality of Section \ref{lowbds}.  This construction forms the main technical step in the arguments asserting Theorems \ref{l1upbd} and \ref{l2upbd}, which we prove in Section \ref{existence}.   In this section we also discuss criteria for the fundamental solutions of $\mathcal{L}$ and $\mathcal{G}$ to be pointwise equivalent to the fundamental solutions of the unperturbed differential operators.

Finally, in Section \ref{regularity}, we consider the Sobolev regularity of the fundamental solution away from its pole.  This is the content of Theorem \ref{fundreg} below.

\section{Main results}\label{main}

  We need to introduce some notation before we can state our results.  The global bounds will involve two local potentials, a nonlinear Wolff potential, and a linear Riesz potential.  If $s>1, \alpha>0$ with $0 < \alpha s < n$, we define the local Wolff potential of a measure $\sigma$, for $\rho>0$, by:
\begin{equation}\label{localwolff}
\mathbf{W}_{\alpha,s}^{\rho} \sigma (x) = \int_0^{\rho} \left (\frac{\sigma (B(x,r))}{r^{n-\alpha s}}\right )^{1/(s-1)}\frac{dr}{r}. 
\end{equation}
For $0<\alpha<n$ the local Riesz potential of $\sigma$ is defined by:
\begin{equation}\label{localriesz}
\mathbf{I}_{\alpha}^{\rho} \sigma (x) = 
\int_0^{\rho} \frac{\sigma (B(x,r))}{r^{n-\alpha}}\frac{dr}{r} .
\end{equation}

We make the convention that when $\rho=+\infty$, then we write $\textbf{W}_{\alpha, s}\sigma$ and $\textbf{I}_{\alpha} \sigma $ for $\textbf{W}_{\alpha, s}^{\infty} \sigma $ and $\textbf{I}_{\alpha}^{\infty} \sigma$ respectively. In particular,
\begin{equation}\label{globalriesz}
\mathbf{I}_{\alpha} \sigma (x)  =\int_0^{+\infty} \frac{\sigma (B(x,r))}{r^{n-\alpha}}\frac{dr}{r}=(n-\alpha)^{-1} \int_{\mathds{R}^n} \frac{d \sigma(y)}{\abs{x-y}^{n-\alpha}} . 
\end{equation}
When $d \sigma= f(x) \, dx$ where $f \in L^1_{loc}(dx)$, we will denote the corresponding 
potentials by 
 $\textbf{W}_{\alpha, s} f$ and $\mathbf{I}_{\alpha} f$ respectively.

 \subsection{} Let us first state our main result for the quasilinear operator $\mathcal{L}$ defined by (\ref{defnl1}).  We choose to work with solutions in the potential theoretic sense, see Section \ref{background} below.  The reader should note that these solutions are by definition lower semicontinuous, and hence defined everywhere, and so it makes sense to talk about pointwise bounds.  We could have alternatively worked with solutions in the 
\textit{renormalized sense}, see \cite{DMMOP99} for a thorough introduction.

\begin{defn}\label{l1fund} A fundamental solution (with pole at $x_0$) of the operator $\mathcal{L}$ defined by (\ref{defnl1}), is a positive $p$-superharmonic function $u(\,\cdot\,, x_0)$, such that $u \in L^{p-1}_{\rm{loc}}(\sigma)$, satisfying equation (\ref{fundsolnintro}).  The equality in (\ref{fundsolnintro}) is understood in the $p$-superharmonic sense.  See Section \ref{background} below for more details.
\end{defn}

When we write \textit{$u(x,x_0)$ is a fundamental solution of $\mathcal{L}$}, with no mention of the pole, we tacitly assume that it has pole at $x_0$.

The first theorem concerns the lower bound for fundamental solutions.  Throughout this paper, unless stated otherwise, we will make the assumption that the measure $\sigma$ is not identically $0$.

\begin{thm}\label{l1lowbd} 
a)  Let $1<p<n$, $x_0 \in \mathds{R}^n$, and suppose $u(\, \cdot\,, x_0)$ is a fundamental solution of $\mathcal{L}$ with pole at $x_0$. Then (\ref{capintro}) holds with $C=1$. In addition, there is a constant $c>0$, depending on $n, p$  such that the bound (\ref{l1lowest}) holds.  In other words, for all $x\in \mathds{R}^n$
$$ u(x, x_0) \geq c \abs{x-x_0}^{\frac{p-n}{p-1}}\exp\Bigl( c\mathbf{W}_{1,p}^{\abs{x-x_0}}(\sigma)(x)+c \mathbf{I}_{p}^{\abs{x-x_0}}(\sigma) (x_0)\Bigl).
$$
b)  If $p\geq n$, and $u$ is a  nonnegative $p$-superharmonic function satisfying the differential inequality:
$$\mathcal{L}u \geq 0, \;\; \text{in}\,\, \mathds{R}^n
$$
then $u \equiv 0$.
\end{thm}

\begin{rem} Part b) of Theorem \ref{l1lowbd} is a Liouville theorem, and when $p>n$ it is related to the important recent works of Serrin and Zou (see \cite{SZ1}, Theorem II'), and Bidaut-V\'{e}ron and Pohozaev \cite{BVP1}.  When $p=n$ the result is a straightforward consequence of well known local estimates of the Riesz measure of a $p$-superharmonic function, for instance one may use Lemma 3.5 in \cite{KM}.  For several special cases the result follows from those in \cite{BVP1}.\end{rem}

\begin{rem}  As we shall see below (in Lemma \ref{plapcap}), the condition (\ref{capintro}) is in fact necessary for the existence of a positive $p$-superharmonic function satisfying the inequality $\mathcal{L} u \geq 0$ in the $p$-superharmonic sense.
\end{rem}

In the case when $1<p\leq n$, it is a nontrivial fact that when $\sigma \equiv 0$ that the fundamental solution is in fact unique; this was proved in \cite{KicVer1}.  An alternative method is outlined in \cite{TWHess3}, where uniqueness of the fundamental solution to the fully nonlinear $k$-Hessian operators when $1\leq k \leq n/2$ is treated.  However, when $\sigma$ is not trivial, it is known even in the linear case ($p=2$, or $k=1$) that solutions of $\mathcal{L}$ are not necessarily unique for a general measure $\sigma$  (see \cite{Mur86}).  It is therefore desirable to single out a distinguished class of fundamental solutions.  We are interested in fundamental solutions of $\mathcal{L}$ which behave like the lower bound (\ref{l1lowest}).  The existence of such fundamental solutions, called \textit{minimal fundamental solutions}, is the content of the next theorem.

\begin{thm}\label{l1upbd}  Let $1<p<n$, $x_0 \in \mathds{R}^n$ and suppose $\sigma$ is a nonnegative Borel measure so that (\ref{capintro}) holds.  There is a constant $C_0 = C_0(n,p)>0$ such that if (\ref{capintro}) holds with constant $C<C_0$, then there exists a fundamental solution $u(\,\cdot, x_0)$ of $\mathcal{L}$ with pole at $x_0$, together with a constant $c=c(n,p,C)>0$, so that the upper bound (\ref{l1upest}) holds for all $x\in \mathds{R}^n$, i.e.
$$ u(x, x_0) \leq c \abs{x-x_0}^{\frac{p-n}{p-1}}\exp\Bigl( c\mathbf{W}_{1,p}^{\abs{x-x_0}}(\sigma)(x)+c \mathbf{I}_{p}^{\abs{x-x_0}}(\sigma) (x_0)\Bigl).
$$
\end{thm}

\begin{rem}As a corollary of Proposition \ref{minfund} - which states that \textit{whenever there exists a fundamental solution of $\mathcal{L}$ with pole at $x_0$, then there exists a unique minimal fundamental solution of $\mathcal{L}$ with pole at $x_0$} - we assert the existence of a unique minimal fundamental solution of (\ref{defnl1}) obeying the bounds (\ref{l1lowest}) and (\ref{l1upest}).  See Corollary \ref{minbnds} below.\end{rem}

When $p=2$, the $p$-Laplacian reduces to the Laplacian operator and Theorems \ref{l1lowbd} and \ref{l1upbd} are contained in some very recent work of M. Frazier and the second author \cite{FV1}.  In fact when $p=2$ the lower bound, Theorem \ref{l1lowbd}, has been known for some time, under various restrictions on $\sigma$ (see \cite{GriHan1}).  The corresponding upper bound seems to be much deeper.  In \cite{FV1}, \cite{FNV} such bounds for the Green function of Schr\"{o}dinger type equations with the fractional Laplacian operator are discussed.  

\begin{rem}  From our method it is clear that Theorems \ref{l1lowbd} and \ref{l1upbd} continue to hold if we replace the $p$-Laplacian operator by the general quasilinear $\mathcal{A}$-Laplacian operator $\rm{div} \, \mathcal{A}(x, \nabla u)$ (see, e.g.,  \cite{HKM}, and Section \ref{background} below).  The constants appearing in the theorems will then in addition depend on the structural constants of  $\mathcal{A}$.
\end{rem}

Having constructed a fundamental solution, we now turn to considering how regular it is away from the pole $x_0$.   This is the content of the next theorem.

\begin{thm}\label{fundreg}  Suppose the hypothesis of Theorem \ref{l1upbd} are satisfied, and that $u(x,x_0) \not\equiv \infty$, with $u(x,x_0)$ the fundamental solution constructed in Theorem \ref{l1upbd}.  Then, there exists $C_0=C_0(n,p)>0$ so that if (\ref{capintro}) holds with $C<C_0$, then:
$$u(\,\cdot, x_0)\in W^{1,p}_{\text{loc}}(\mathds{R}^n\backslash\{x_0\}).
$$
\end{thm}

\begin{rem} \label{fundrem}The local Sobolev regularity $W^{1,p}_{\text{loc}}(\mathds{R}^n\backslash \{x_0\})$ is optimal for solutions of $\mathcal{L}(u)=0$ under the assumption (\ref{capintro}) on $\sigma$, see \cite{JMV10}.  Theorem \ref{fundreg} seems to be new in the linear case $p=2$. In this case the proof, given in Section \ref{regularity}, can clearly be easily adapted to deduce the local regularity of the minimal Green's function of the Schr\"{o}dinger operator in a bounded domain $\Omega$, as was constructed recently in \cite{FV1, FNV}.
\end{rem}

\subsection{} We now move onto a fully nonlinear analogue of Theorems \ref{l1lowbd} and \ref{l1upbd}.  Let $1\leq k \leq n$ be an integer. Then the second operator we consider, denoted by $\mathcal{G}$, is the fully nonlinear operator defined by:
\begin{equation}\label{defn2}
\mathcal{G}(u) = F_k(-u) - \sigma \abs{u}^{k-1}u.
\end{equation}
Here $\sigma$ is again a nonnegative Borel measure, and $F_k$ is the $k$-Hessian operator, introduced by Caffarelli, Nirenberg and Spruck \cite{CNS1}, and defined for smooth functions $u$ by:
$$F_k(u) = \sum_{1\leq i_1 < \cdots < i_k \leq n} \lambda_{i_1}\dots \lambda_{i_k}$$
with $\lambda_1, \dots \lambda_n$ denoting the eigenvalues of the Hessian matrix $D^2u$.  We will use the notion of $k$-convex  functions, introduced by Trudinger and Wang \cite{TW2}, to state our results.  See Section \ref{background} for a brief discussion and definitions.

\begin{defn}\label{l2fund}  A fundamental solution (with pole at $x_0$) $u( \, \cdot\,,x_0)$ of $\mathcal{G}$ is a function such that $-u(\, \cdot, x_0)$ is a $k$-convex function so that $u(\, \cdot, x_0) \in L^k_{\text{loc}}(\sigma)$ satisfying  $\displaystyle \mathcal{G} u (\, \cdot, x_0) = \delta_{x_0}$ in the viscosity sense, and $\displaystyle \inf_{x\in \mathds{R}^n}u(x, x_0) = 0$.
\end{defn}

The necessary condition on $\sigma$ is now considered in terms of the $k$-Hessian capacity, introduced in \cite{TW1};
\begin{equation}\label{Hesscap}
\text{cap}_k (E) = \sup \{\;\; \mu_k[u](E) \; : \; u \text{ is } k\text{-convex in } \mathds{R}^n, \; -1<u<0 \;\; \},
\end{equation}
for a compact set $E$.  Here $\mu_k[u]$ is the $k$-Hessian measure of $u$; see Theorem \ref{Hessweakcont} below.
\begin{thm}\label{l2lowbd}
a)  Let $1\leq k < n/2$, and let $x_0 \in \mathds{R}^n$.  If $u(\, \cdot, x_0)$ is a fundamental solution of $\mathcal{G}$, then there is a constant $C>0$, $C=C(n,k)$, such that 
\begin{equation}\label{l2capest}\sigma(E) \leq C \, \rm{cap}_k(E) \;\; \text{ for all compact sets } E\subset \mathds{R}^n.
\end{equation}
In addition, there is a constant $c>0$, $c=c(n,k, C)$, such that 
\begin{equation}\label{l2lowest}\begin{split}
u(x, x_0) \geq c\abs{x-x_0}^{2-\frac{n}{k}}& \exp\Big( c\int_{0}^{\abs{x-x_0}}\Bigl(\frac{\sigma(B(x,r)}{r^{n-2k}}\Bigl)^{1/k}\frac{dr}{r}\Bigl)\\
&\cdot \exp\Bigl(c\int_{0}^{\abs{x-x_0}} \frac{\sigma(B(x_0,r))}{r^{n-2k}}\frac{dr}{r}\Bigl). 
\end{split}\end{equation}
b)  Let $k \geq n/2$. Then if $u$ is a nonnegative function so that $-u$ is a $k$-convex function satisfying the inequality:
$$\mathcal{G}(u) \geq 0 \quad \text{in} \;\; \mathds{R}^n
$$
then $u\equiv 0$.
\end{thm}

\begin{thm}\label{l2upbd}  Let $1\leq k<n/2$, and suppose $\sigma$ is a nonnegative Borel measure satisfying (\ref{l2capest}).  There is a constant $C_0 = C_0(n,k)$, such that if $C<C_0$ and (\ref{l2capest}) holds with constant $C$, then there exists a fundamental solution $u(\,\cdot, x_0)$ of $\mathcal{G}$, together with a constant $c=c(n,k,C)$ so that 
\begin{equation}\begin{split}\label{l2upest}
u(x, x_0) \leq  c\abs{x-x_0}^{2-\frac{n}{k}}& \exp\Big( c\int_{0}^{\abs{x-x_0}}\Bigl(\frac{\sigma(B(x,r)}{r^{n-2k}}\Bigl)^{1/k}\frac{dr}{r}\Bigl)\\
&\cdot \exp\Bigl(c\int_{0}^{\abs{x-x_0}} 
\frac{\sigma(B(x_0,r))}{r^{n-2k}}\frac{dr}{r}\Bigl) .
\end{split}\end{equation}
\end{thm}

\begin{rem}  Part b) of Theorem \ref{l2lowbd} is easy to see using well known local estimates.  For instance, one can readily deduce the result from \cite{TW2}, Theorem 3.1, along with a routine approximation argument using weak convergence of Hessian measures. 
\end{rem}

\section{Preliminaries}\label{background}

\subsection{Notation}  For an open set $\Omega$, we denote by $L^p_{\text{loc}}(\Omega)$ to be the space of functions locally integrable to the $p$-th power with respect to Lebesgue measure.  Similarly, $L_{\text{loc}}^p(\Omega, d\sigma)$ then denotes the space of functions which are locally integrable to the $p$-th power with respect to $\sigma$ measure. $W^{1,p}_{\text{loc}}(\Omega)$ is the space of functions $u\in L^p_{\text{loc}}(\Omega)$, with weak derivative $\nabla u \in L^p_{\text{loc}}(\Omega; \mathds{R}^n)$.  Finally, we will use the symbol $A \lesssim B$ to mean that $A\leq CB$ with the constant $C>0$ depending on the allowed parameters of the particular result being proved.

\subsection{}In this section we will introduce some fundamental results from the potential theory of nonlinear elliptic equations.  Two results will be key to our study: a potential estimate; and a weak continuity result.  The potential which the estimates will involve is called the Wolff potential \cite{HW}.  For $s>1$ and $0< \alpha s <n$, we define the Wolff potential of a nonnegative Borel measure $\mu$ by:
\begin{equation}\label{Wolff}\mathbf{W}_{\alpha, s}\mu (x) = \int_0^{\infty}\Bigl( \frac{\mu(B(x,r))}{r^{n-\alpha s}}\Bigl)^{1/(s-1)}\frac{dr}{r}
\end{equation}

We first will discuss quasilinear equations. The material regarding these equations is drawn from \cite{HKM, KM, KM1, PV, PV2, TW1, MZ1}. 

Let us assume that $\mathcal{A}: \mathds{R}^{n}\,\mathrm{x}\,\mathds{R}^{n} \rightarrow \mathds{R}^n$ satisfies:
$$ x\rightarrow \mathcal{A}(x, \xi) \;\; \text{is measurable for all } \xi\in \mathds{R}^n, \text{ and }$$
$$ \xi\rightarrow \mathcal{A}(x, \xi) \;\; \text{is continuous for a.e. } x \in \mathds{R}^n. $$
In addition suppose that there are constants $0<\alpha \leq \beta < \infty$ so that for a.e. $x\in \mathds{R}^n$:
$$\alpha \abs{\xi}^p \leq \mathcal{A}(x,\xi)\cdot \xi, \;\text{ and }\; \abs{\mathcal{A}(x, \xi)} \leq \beta \abs{\xi}^{p-1} .$$
We will also assume that:
$$( \mathcal{A}(x, \xi_1) - \mathcal{A}(x, \xi_2) ) \cdot (\xi_1 - \xi_2) > 0$$
whenever $\xi_1 \neq \xi_2$.

Now, let $\Omega$ be an open subset of $\mathds{R}^n$,  (we will be most interested in the case $\Omega = \mathds{R}^n$). Whenever $u \in \mathrm{W}^{1,p}_{\text{loc}}(\Omega)$, we define the divergence of $\mathcal{A}(x, \nabla u)$ in the distributional sense. 
As follows from the classical regularity theory of de Giorgi, Nash and Moser, any $u \in W^{1,p}_{\text{loc}}(\Omega)$ solution of $-\text{div}\mathcal{A}(x, \nabla u) = 0$ in the distributional sense has a locally H\"{o}lder continuous representative, and we call these representatives $\mathcal{A}$-harmonic functions.  Here and in the following the $p$-Laplacian operator corresponds to the choice of $\mathcal{A}(x, \xi) = \abs{\xi}^{p-2}\xi$, in this case $\mathcal{A}$-harmonic functions are called $p$-harmonic functions, and similarly $p$-superharmonic functions are $\mathcal{A}$-superharmonic functions (as defined below) in this special case. 

In analogy with classical  superharmonic functions, we define the $\mathcal{A}$-superharmonic functions via a comparison principle. We say that $u: \Omega \rightarrow (-\infty, \infty]$ is $\mathcal{A}$-superharmonic if $u$ is lower semicontinuous, is not identically infinite in any component of $\Omega$, and satisfies the following comparison principle:  Whenever $D\subset\subset \Omega$ and $h \in C( \bar{D})$ is $\mathcal{A}$-harmonic in $D$, with $h \leq u$ on $\partial D$, then $h \leq u$ in $D$. 

An $\mathcal{A}$-superharmonic function $u$ does not necessarily have to belong to $\mathrm{W}^{1,p}_{\text{loc}}(\Omega)$, but its truncates $T_k(u) = \min (u,k) \in \mathrm{W}^{1,p}_{\text{loc}}(\Omega)$ for all $k>0$.   In addition $T_k(u)$ are supersolutions, i.e. $-\text{div}\mathcal{A}(\,\cdot, \nabla T_k(u)) \geq 0$, in the distributional sense (see \cite{HKM}). 

The above paragraph leads us to the definition of the generalized gradient of an $\mathcal{A}$-superharmonic function $u$ as:
$$Du = \lim_{k\rightarrow \infty} \nabla (T_k(u)) .$$

\begin{rem} There are alternative notions of solutions which we could have introduced to obtain our results, for instance either \textit{renormalized solutions} or \textit{supersolutions up to all levels}, see \cite{DMMOP99} and \cite{MZ1} respectively.  We chose to use the language of $\mathcal{A}$-superharmonic functions because Theorems \ref{weakcont} and \ref{kmpotest} were developed in this framework.
\end{rem}

Let $u$ be $\mathcal{A}$-superharmonic and let $1\leq q < n/(n-1)$. Then it is proved in \cite{KM} that $\abs{Du}^{p-1}$ and $\mathcal{A}(\cdot, Du)$ belong to $\mathrm{L}^q_{loc}(\Omega)$.  This allows us to define a nonnegative distribution for each  $\mathcal{A}$-superharmonic function $u$ by:
\begin{equation}
-\text{div}\mathcal{A}(x, \nabla u) (\psi) = \int_{\Omega} \mathcal{A}(x, Du)\cdot \nabla \psi \; dx
\end{equation}
for $\psi \in C^{\infty}_{0}(\Omega)$.  So, the Riesz representation theorem yields the existence of a unique nonnegative Borel measure $\mu[u]$ so that $-\text{div}\mathcal{A}(x, \nabla u) = \mu[u]$.  Furthermore, by the integrability of the gradient, it follows that for any $r>n$:
\begin{equation}\label{widerclass}\int_{\Omega} \mathcal{A}(\, \cdot, Du) \cdot\nabla \phi dx = \int_{\Omega} \phi d\mu, \text{ for all } \phi\in W^{1,r}(\Omega) \text{ with compact support.}
\end{equation}
For a nonnegative measure $\omega$ we will say that $-\text{div}\mathcal{A}(\cdot, \nabla u) = \omega$ in the $p$-superharmonic sense if $u$ is $p$-superharmonic, and $\mu[u] = \omega$.  Thus $\mathcal{L}(u) = \omega$ in the $p$-superharmonic sense if $\mu[u] = \sigma u^{p-1} + \omega$. 

We now state a very useful convergence result, contained in Kileplainen and Maly \cite{KM}, Theorem 1.17.

\begin{thm}\label{limsuper} \cite{KM} Suppose $\{u_j\}_j$ is a sequence of nonnegative $\mathcal{A}$-superharmonic functions in an open set $\Omega$.  Then there is a subsequence $\{u_{j_k}\}_k$ which converges almost everywhere to a nonnegative function $u$ which is either $p$-superharmonic or identically infinite in each component of $\Omega$.
\end{thm}

The next result, first stated explicitly in \cite{TW1}, shows that $\mathcal{A}$-Laplace operator is weakly continuous. 

\begin{thm} \cite{TW1} \label{weakcont} Suppose $\{u_j\}_j$ is a sequence of nonnegative $\mathcal{A}$-superharmonic functions which converge almost everywhere to an $\mathcal{A}$-superharmonic function $u$.  Then $\mu [u_j]$ converges weakly to $\mu[u]$.
\end{thm}

The second major result we need is the Wolff's potential estimates of Kilpel\"{a}inen and Maly \cite{KM1} (see also \cite{MZ1}, \cite{PV}). 

\begin{thm}\cite{KM1}\label{kmpotest} Let $u$ be a nonnegative $\mathcal{A}$-superharmonic function in $\mathds{R}^n$ so that $\inf_{x\in \mathds{R}^n} u(x) =0$.  If $\mu = -\rm{div}\,\mathcal{A}(\cdot\,, \nabla u)$, then there is a constant $K = K(n,p, \alpha, \beta)$, so that for all $x\in \mathds{R}^n$, 
\begin{equation}
\frac{1}{K} \,  \mathbf{W}_{1,p} \mu (x)\leq u(x) \leq  K \, \mathbf{W}_{1,p} \mu (x) .
\end{equation}
\end{thm}

\subsection{}We now turn to the fully nonlinear counterpart of these results. A very recent and comprehensive account of the $k$-Hessian equation is \cite{Wang1}.  Here $k$-convex functions associated to the $k$-Hessian operator, introduced by Trudinger and Wang \cite{TW2}, will play the role of $\mathcal{A}$-superharmonic functions in the quasilinear theory above.  Let $\Omega \subset \mathds{R}^n$ be an open set, let $k=1, \dots, n$ and $u\in C^2(\Omega)$, then the $k$-Hessian operator is:
$$F_k(u) = \sum_{1\leq i_1< \dots <i_k \leq n}\lambda_{i_1}\dots \lambda_{i_k}
$$
where $\lambda_1, \dots, \lambda_n$ are the eigenvalues of the matrix $D^2 u$.  We will then say that $u$ is $k$-convex in $\Omega$ if $u:\Omega \rightarrow [-\infty, \infty)$ is upper semicontinuous and satisfies $F_k(u)\geq 0$ in the viscosity sense, i.e. for any $x\in \Omega$, $F_k(q)(x)\geq 0$ for any quadratic polynomial $q$ so that $u-q$ has a local finite maximum at $x$.  Equivalently (see \cite{TW2}), we may define $k$-convex functions by a comparison principle: an upper semicontinuous function $u : \Omega \rightarrow [-\infty, \infty)$ is $k$-convex in $\Omega$ if for every open set $D \subset\subset \Omega$, and $v\in C^2_{\text{loc}}(D)\cap C(\bar{D})$ with $F_k(v)\geq 0$ in $D$, then 
$$ u\leq v \;\text{on}\;\partial D \implies \; u\leq v \; \text{in}\; D .$$ 
Let $\Phi ^k(\Omega)$ be the set of $k$-convex functions such that $u$ is not identically infinite in each component of $\Omega$.  The following weak continuity result is key to us. 
\begin{thm}\cite{TW2}\label{Hessweakcont} Let $u\in \Phi^k(\Omega)$. Then there is a  nonnegative Borel measure $\mu_k[u]$ in $\Omega$ such that 
\begin{itemize}
\item $\mu_k[u] = F_k(u)$ whenever $u\in C^2(\Omega)$, and
\item If $\{u_m\}_m$ is a sequence in $\Phi^k(\Omega)$ converging in $L^1_{\text{loc}}(\Omega)$ to a function $u$, then the sequence of measures $\{\mu_k[u_m]\}_m$ converges weakly to $\mu_k[u]$. 
\end{itemize}
\end{thm}

The measure $\mu_k[u]$ associated to $u\in \Phi^k(\Omega)$ is called the \textit{Hessian measure} of $u$.  Hessian measures were used by Labutin \cite{Lab1} to deduce Wolff's  potential estimates for a $k$-convex function in terms of its Hessian measure.  The following global version of Labutin's estimate is deduced from his result in \cite{PV}:
\begin{thm}\cite{PV} \label{Hesswolff}
Let $1\leq k \leq n$, and suppose that $u\geq 0$ is such that $-u\in \Phi^k(\Omega)$ and $\inf_{x\in \mathds{R}^n}u(x) =0$.  Then, if $\mu=\mu_k[u]$, there is a positive constant $K$, depending on $n$ and $k$, such that:
$$c_1\mathbf{W}_{\frac{2k}{k+1}, k+1}\mu(x) \leq u(x) \leq c_2 \mathbf{W}_{\frac{2k}{k+1}, k+1}\mu(x), \qquad x \in \mathds{R}^n.
$$
\end{thm}

\subsection{} This subsection is concerned with minimality of fundamental solutions.  A \textit{minimal fundamental solution} $u(x,x_0)$ of $\mathcal{L}$ defined by (\ref{defnl1}), is a fundamental solution of $\mathcal{L}$ as in Definition \ref{l1fund}, so that $u(x,x_0) \leq v(x,x_0)$ whenever $v(x,x_0)$ is a fundamental solution of $\mathcal{L}$.  Our aim is to prove the following proposition.  

\begin{prop} \label{minfund} Let $1<p<n$ and $\sigma$ be a nonnegative measure.  Suppose that there exists a fundamental solution $v(x,x_0)$ of $\mathcal{L}$ with pole at $x_0$.  Then there exists a unique minimal fundamental solution $u(x,x_0)$ of $\mathcal{L}$.
\end{prop}

We will need the following simple lemma, and as we could not locate a reference we will provide a proof.
\begin{lem} \label{solnlessthan} Let $\Omega \subset \mathds{R}^n$ be a bounded Lipschitz domain, and suppose that $v$ is a positive $p$-superharmonic in $\Omega$ so that $T_k(u) \in W^{1,p}(\Omega)$ for all $k>0$, and $-\Delta_p v = \nu$.  Let $\mu \leq \nu$, be a compactly supported measure in $\Omega$, then there is a nonnegative $p$-superharmonc fuction $w$, such that $w\leq v$ and:
\begin{equation}\label{wequation}-\Delta_p w = \mu  \text{ in } \,\Omega, \;\; w = 0 \text{ continuously on } \partial \Omega.
\end{equation}
\end{lem}
\begin{proof}  Let $T_k(v) = \min(v,k)$, and let $\nu_k$ be the Riesz measure of $T_k(v)$.  Then $\nu_k \in W^{-1, p'}(\Omega)$, and $\nu_k \rightarrow \nu$ weakly.  Let $\mu_k$ be a sequence in $W^{-1, p'}(\Omega)$ so that $\mu_k \leq \nu_k$ and $\mu_k \rightarrow \mu$ weakly.  By the compact support of $\mu$ we may also assume that there is a compactly supported set $K \subset \Omega$, which contains the support of $\mu_k$, for each $k$ (otherwise we just multiply $\mu_k$ by a smooth bump function $\phi \in C^{\infty}_{0}(K)$ such that $\phi \equiv 1$ on the support of $\mu$).  Let $w_k \in W^{1,p}(\Omega)$ be the solution of:
$$-\Delta_p w_k = \mu_k \text{ in } \,\Omega, \;\; w_k =0 \text{ on } \partial \Omega.$$
Such a unique solution exists by the theory of monotone operators, see e.g. \cite{Li69}.  In addition, $w_k \leq v_k \leq v$ in $\Omega$ by the classical comparison principle.  Therefore, by \cite{KM}, Theorem 1.17, we see that by a relabeling of the sequence, we may assert that there is a $p$-superharmonic function $w = \lim_{k\rightarrow \infty} w_k$ almost everywhere, with $w \leq v$ and $-\Delta_p w = \mu$.

It remains to prove that $w$ is zero at the boundary and attains its boundary value continuously.  First note that each $w_k$ is $p$-harmonic in $\Omega\backslash K$. Since $\Omega$ is Lipschitz, there exists $M\geq 2$,  $c>0$ and $0<r_0 < d(K, \partial \Omega)/4$, such that for all $z\in \partial \Omega$ and $0<r<r_0$: $\sup_{B(z, r/c)\cap\Omega}w_k \leq c\, w_k(a(z))$, here $a(z)$ is a point such that $M^{-1}r \leq \abs{a(z)-z} \leq M r$.   This is a well known  boundary estimate, see e.g. \cite{BVBL, LN1}.   Combined with the boundary regularity of $p$-harmonic functions, \cite{MazBd} (see also \cite{MZ1, HKM}), we see that each $w_k$ is locally H\"{o}lder continuous in a neighbourhood of each boundary point with constants independent of $k$.   Indeed, there exists constants $c, \theta >0$ depending on $n$ and $p$, such that if $0<r<r_0$, then for each $z \in \partial \Omega$ and $x, y \in B(z, r/c)\cap \Omega$:
\begin{equation}\begin{split}\label{holdcont} \abs{w_k (x) - w_k(y)} & \leq c\max_{B(w, r/c)\cap\Omega}w_k \cdot \abs{x-y}^{\theta} \leq c\, w_k(a(z))  \cdot \abs{x-y}^{\theta}\\
&\leq c  \inf_{B(a(z) , r/2M)}w_k \cdot \abs{x-y}^{\theta}  \leq c \inf_{B(a(z) , r/2M)}v \cdot \abs{x-y}^{\theta}.
\end{split}\end{equation}
The third inequality in display (\ref{holdcont}) follows from the second by Harnack's inequality. That $w=0$ continuously on $\partial \Omega$ follows from (\ref{holdcont}).
\end{proof}

By Theorem \ref{l1lowbd}, we may assume that $\sigma$ satisfies (\ref{capintro}) (see Lemma \ref{plapcap} below), in proving Proposition \ref{minfund}.  This assumption is the key for the construction, as we will apply uniqueness results.  For general measure data, the uniqueness of solutions in a suitable sense is an open problem for the $p$-Laplacian.

\begin{proof}[Proof of Proposition \ref{minfund}]  Let $w$ be any fundamental solution of the operator $\mathcal{L}$ defined by (\ref{defnl1}) with pole at $x_0$.  We will construct a fundamental solution $u$ so that $u \leq w$.  This construction will be independent of choice of $w$ and hence will prove the proposition.
Our first goal is to show $w\geq u_0 : = G(\cdot,x_0)$, with $G(x,x_0)$ defined as in (\ref{unperturbed}).  By using Lemma \ref{solnlessthan} repeatedly in a sequence of concentric balls, along with Theorems \ref{limsuper} and \ref{weakcont}, we assert the existence of a solution $w_0$ of $-\Delta_p w_0 = \delta_{x_0}$ in $\mathds{R}^n$, with $w_0 \leq w$, and hence $\inf_{x\in \mathds{R}^n}w_0(x) = 0$.  Since $G(x,x_0)$ is unique (see \cite{KicVer1}), it follows that $w_0 = u_0$.  Thus $w\geq u_0$.

Now suppose that $w\geq u_{m-1}$.  Then, for each $j$ and $k>j$, we see by Lemma \ref{solnlessthan} there is a positive $p$-superharmonic function $u_m^{j,k}$ solving:
$$-\Delta_p u_m^{j,k} = (\sigma u_{m-1}^{p-1})\chi_{B(x_0, 2^{j})} + \delta_{x_0} \;\; \text{ in } \, B(x_0, 2^k), \;\; u_m^{j,k} = 0\; \text{ on } \; \partial B(x_0, 2^k)
$$
with $u_m^{j,k} \leq w$.  But using Theorem 4.2 of \cite{TWQuas} (which applies as a simple consequence of (\ref{capintro}), and that $u_m^{k,j}$ being $p$-harmonic near $\partial B(x_0, 2^k)$), we see that $u_m^{j,k}$ is unique (and hence independent of $w$).  By combining Theorems \ref{limsuper} and \ref{weakcont}, we conclude that there exists a $p$-superharmonic function $u_m^j$ such that $-\Delta_p u_m^j = (\sigma u_{m-1}^{p-1})\chi_{B(x_0, 2^{j})} + \delta_{x_0}$ in $\mathds{R}^n$.  Furthermore $u_m^j \leq w$, and hence $\inf_{x\in \mathds{R}^n} u_m^j(x) = 0$.  We remark here that there are other uniqueness results, (for instance see \cite{DMMOP99}) which could very probably be used, but the cited theorem above is quickest to verify with our notion of solution.

Again by Theorem \ref{limsuper}, and weak continuity (Theorem \ref{weakcont}), there exists a $p$-superharmonic function $u_m$ such that:  $-\Delta_p u_m = \sigma u_{m-1}^{p-1} + \delta_{x_0}$ in $\mathds{R}^n$ and $u_m \leq w$.  Therefore $\inf_{x\in \mathds{R}^n}u_m(x) = 0$.  Appealing to Theorem \ref{limsuper} and weak continuity a final time, we find a $p$- superharmonic function $u$ such that $-\Delta_p u = \sigma u^{p-1} + \delta_{x_0}$ in $\mathds{R}^n$
and $u \leq w$, thus $\inf_{x\in \mathds{R}^n} u(x) =0$ and $u$ is a fundamental solution of $\mathcal{L}$.

The proposition is proved, since whenever $w$ is a fundamental solution of $\mathcal{L}$, then iteratively we see that $w\geq u_m$ for all $m$ and hence $w\geq u$.
\end{proof}
With this proposition the following Corollary is an immediate consequence of Theorems \ref{l1lowbd} and \ref{l1upbd}.
\begin{cor}\label{minbnds}  Suppose that $\sigma$ is a nonnegative measure satisfying (\ref{capintro}) with constant $C>0$.  Then there exists a positive constant $C_0$ depending on $n$ and $p$, so that if $C<C_0$, there exists a unique minimal fundamental solution $u(x,x_0)$ of $\mathcal{L}$ defined by (\ref{defnl1}).  Furthermore $u(x,x_0)$ satisfies global bilateral bounds (\ref{l1lowest}) and (\ref{l1upest}), with a different constant $c = c(n,p)>0$ in each direction.
\end{cor}

The existence of a minimal fundamental solution for the $k$-Hessian operators can be shown in a similar way to the quasilinear case presented above, adapting techniques in \cite{TWHess3}.

\subsection{}We finish this section with a brief discussion of capacity.  In the range of exponents we are interested in, both the $p$-capacity and the $k$-Hessian capacities are equivalent, for compact sets, with certain Riesz capacities.

Let $s>1$ and $0<\alpha<n$. For $E\subset \mathds{R}^n$, we define the Riesz capacity of $E$ by the following:
\begin{equation}\label{capdefn}\text{cap}_{\alpha,s}(E) = \inf \{ \,\norm{f}^s_{L^s} \; : \;\; f \in L^s(\mathds{R}^n), \, \,  f \ge 0, \;\; \mathbf{I}_{\alpha} f\geq 1  \;\text{on}\; E\;\;\}. 
\end{equation}
See (\ref{globalriesz}) for the definition of the Riesz potential $\mathbf{I}_{\alpha}$.

Recall the $p$-capacity defined in (\ref{pcapintro}). Then we have the following equivalence.

\begin{lem} \label{quascapequiv} Let $1<p<n$. Then there is a positive constant $C=C(n,p)$ so that, for all compact sets $E\subset \mathds{R}^n$:
$$\frac{1}{C}\rm{cap}_{1,p}(E) \leq \rm{cap}_p(E) \leq C \, \rm{cap}_{1,p}(E). 
$$
\end{lem}
For a proof of this Lemma, see, e.g., \cite{MazSob} or \cite{MZ1}.

Now, recall the $k$-Hessian capacity (\ref{Hesscap}).  Then the  following equivalence holds (see  Theorem 2.20 in \cite{PV}).

\begin{lem} \label{hesscapequiv}  Let $1\leq k < n/2$. Then there is a positive constant $C=C(n,k)>0$ so that for all compact sets $E\subset \mathds{R}^n$:
$$\frac{1}{C} \rm{cap}_{\frac{2k}{k+1}, k+1}(E) \leq \rm{cap}_k(E) \leq 
C\, \rm{cap}_{\frac{2k}{k+1}, k+1}(E).
$$
\end{lem}

\section{Reduction to integral inequalities and necessary conditions on $\sigma$}\label{reduction}

\subsection{}In this section we will show how our study of the fundamental solutions of $\mathcal{L}$ and $\mathcal{G}$ can be rephrased into a question of nonlinear integral operators.  The Wolff potential estimate will be the key to this idea, recall the definition from (\ref{Wolff}).

Let us introduce two nonlinear integral operators, $\mathcal{N}_1$ and $\mathcal{N}_2$, acting on non-negative functions $f\geq 0$ by:
\begin{equation}\begin{split}\label{N1defn}
\mathcal{N}_1(f)(x) & := \mathbf{W}_{1,p}(f^{p-1}d\sigma)(x), \;\text{    and:}
\end{split}
\end{equation}
\begin{equation}\begin{split}\label{N2defn}
\mathcal{N}_2(f)(x) & := \mathbf{W}_{\frac{2k}{k+1},k+1}(f^{k}d\sigma)(x)
\end{split}
\end{equation}
see also (\ref{Ndefn}) below.  These operators appear naturally in studying the equations $\mathcal{L}(u) = \omega$ and $\mathcal{G}(u) = \omega$ for a nonnegative Borel measure $\omega$.  Indeed, if $1<p<n$ and $u$ is a nonnegative $p$-superharmonic function such that $\mathcal{L}(u) = \omega$, then by the Wolff potential estimate, Theorem \ref{kmpotest}, there is a constant $C=C(n,p)>0$ such that 
$$u(x) \geq C\mathbf{W}_{1,p}(u^{p-1}d\sigma)(x) + C\mathbf{W}_{1,p}(\omega)(x). 
$$ 
Note that from this it follows that $u \in L^{p-1}_{\text{loc}}(\sigma)$.  Hence, if $u$ is a fundamental solution of $\mathcal{L}$, then it follows: 
\begin{equation}\label{N1ineq}
u(x) \geq C \mathcal{N}_1(u)(x) + C\abs{x-x_0}^{\frac{p-n}{p-1}}
\end{equation}
since $\mathbf{W}_{1,p}(\delta_{x_0})(x) = c(n,p)\abs{x-x_0}^{\frac{p-n}{p-1}}$ when $1<p<n$.  Here $C$ is a positive constant depending on $n,p$. 

In much the same way, if $1\leq k<n/2$ and $u$ is a nonnegative function so that $-u$ is a $k$-convex solution of $\mathcal{G}(u) = \omega$ in the sense of $k$-Hessian measures, then by the Wolff potential estimate, Theorem \ref{Hesswolff}, there is a constant $C=C(n,k)>0$ such that 
$$u(x) \geq C\mathcal{N}_2(u)(x) + C\mathbf{W}_{\frac{2k}{k+1}, k+1}(\omega)(x). 
$$
Thus $u\in L^k_{\text{loc}}(\sigma)$, and hence if $u$ is a fundamental solution of $\mathcal{G}$, then there is a constant $C=C(n,k)$ so that 
\begin{equation}\label{N2ineq}
u(x) \geq C\mathcal{N}_2(u)(x) + C\abs{x-x_0}^{2-n/k}. 
\end{equation}

With the aid of the Wolff potential, by introducing the $\mathcal{N}_1$ and $\mathcal{N}_2$, we have rephrased the problem of finding lower bounds for the fundamental solutions to finding lower bounds of solutions of the nonlinear integral inequalities (\ref{N1ineq}) and (\ref{N2ineq}). 

In addition, we will see in Section \ref{construct} that explicitly constructing solutions of (\ref{N1ineq}) and (\ref{N2ineq}) will be the main technical step in proving existence of minimal fundamental solutions of the differential operators $\mathcal{L}$ and $\mathcal{G}$. 

As a result of this discussion it makes sense to introduce a more general nonlinear operator which generalizes both $\mathcal{N}_1$ and $\mathcal{N}_2$.  To this end, recall that the Wolff potential acting on a measure $\omega$ is given by (\ref{Wolff}).

Let $s>1$, $\alpha >0$ so that $0<\alpha s<n$, then we define the nonlinear operator $\mathcal{N}$, for a Borel measurable function $f\geq 0$, by:
\begin{equation}\label{Ndefn}\begin{split}
\mathcal{N}(f)(x) & = \mathbf{W}_{\alpha, s}(f^{s-1}d\sigma)(x)\\
& = \int_0^{\infty} \Bigl(\frac{1}{r^{n-\alpha s}} \int_{B(x,r)} f^{s-1}(z)d\sigma(z) \Bigl)^{1/(s-1)}\frac{dr}{r}
\end{split}\end{equation}
The operators $\mathcal{N}_1$ and $\mathcal{N}_2$ are clearly special cases 
of $\mathcal{N}$ for certain choices of $\alpha$ and $s$.

\subsection{}Fix  $s>1$ and $\alpha$ so that, $0<\alpha s<n$. For the remainder of this section we will be concerned with positive solutions $u$ of the integral inequality:
\begin{equation}\label{genintineq}
u(x) \geq C_0 \mathcal{N}u(x)
\end{equation}
where $C_0$ is a positive constant.  Our first goal will be to prove some necessary conditions on the measure $\sigma$ for there to exist positive solutions of (\ref{genintineq}). In particular, we will prove the following theorem.  Recall the definition of the capacity in (\ref{capdefn}).

\begin{thm}\label{gensigma}  Suppose that $u$ is a positive solution of the inequality (\ref{genintineq}) with constant $C_0>0$.  Then, there is a positive constant $C$, depending on $\alpha, s, n$ and $C_0$, so that for every compact set $E\subset \mathds{R}^n$
\begin{equation}\label{capest}
\sigma(E) \leq C \, \rm{cap}_{\alpha, s}(E).
\end{equation}
\end{thm}

\begin{rem}\label{caprem}  Theorem \ref{gensigma} implies the capacity estimates which appear in Theorems \ref{l1lowbd} and \ref{l2lowbd}.\end{rem}
\begin{proof}[Proof of Remark \ref{caprem}]Suppose first that $u$ is a fundamental solution of $\mathcal{L}$.  Then $u$ satisfies (\ref{N1ineq}), and hence $u$ satisfies (\ref{genintineq}) with $\mathcal{N}=\mathcal{N}_1$.  This corresponds to taking $\alpha =1$ and $s=p$ in the definition of $\mathcal{N}$.  Hence Theorem \ref{gensigma} implies that there is a constant $C>0$ so that $\sigma(E) \leq C \, \rm{cap}_{1,p}(E)$ for all compact sets $E$.  By Lemma \ref{quascapequiv}, this is equivalent to the required capacity estimate in Theorem \ref{l1lowbd}. 

Similarly, if $u$ is a fundamental solution of $\mathcal{G}$, then $u$ satisfies (\ref{N2ineq}), which is the same as (\ref{genintineq}) with $\alpha = \frac{2k}{k+1}$ and $s=k+1$.  Hence Theorem \ref{gensigma} asserts the existence of a constant $C>0$ so that $\sigma(E) \leq C \, \rm{cap}_{\frac{2k}{k+1}, k+1}(E)$ for all compact sets $E$.  Appealing to Lemma \ref{hesscapequiv}, we see that this is equivalent to the capacity condition appearing in Theorem \ref{l2lowbd}.\end{proof}

The same proof shows that Theorem \ref{gensigma} in fact implies the same capacity estimates for any positive solutions of the differential inequalities $\mathcal{L}u \geq 0$ and 
$\mathcal{G}(u) \geq 0$.

\subsection{}We will now briefly discuss an alternative approach to the capacity estimate (\ref{capintro}) in the case of the $p$-Laplacian operator.  This approach was shown to the second author by T. Kilpel\"{a}inen in 1997.

\begin{lem}\label{plapcap} Let $\Omega$ be an open set in $\mathds{R}^n$, and let $\sigma$ be a nonnegative Borel measure absolutely continuous with respect to $p$-capacity. Suppose that $u$ is a positive $p$-superharmonic function such that $-\Delta_p u \geq \sigma u^{p-1}$ in $\Omega$.  Then then following embedding inequality holds:
\begin{equation}\label{embineq}
\int_{\Omega} h^p \, d\sigma \leq \int_{\Omega} \abs{\nabla h}^p dx, \text{  for all } h \in C_0^{\infty}(\Omega), \, h \ge 0, 
\end{equation} 
\end{lem}

\begin{proof} Let $h\geq 0$, $h \in C^{\infty}_{0}(\Omega)$,  Let $\mu[u]$ be the Riesz measure of $u$ (see Section \ref{background}), and $\mu_k$ be the Riesz measure of $T_k(u) =  \min(u,k) \in W^{1,p}_{\text{loc}}(\Omega)$.  It follows that $\mu_k \in W^{-1,p'}_{\text{loc}}(\Omega)$.  Let us decompose $\mu_k$ as:
$$d\mu_k = u^{p-1}d\nu_k + d\omega_k,$$
with $d\nu_k = u^{1-p}\chi_{\{u<k\}}d\mu_k$, and $d\omega_k = \chi_{\{u\geq k\}} d\mu_k$.  This decomposition follows from the minimum principle, since for any compact set $K\subset\subset \Omega$, there exists a constant $c>0$ such that $u\geq c>0$ on $K$.
Since $\mu_k$ lies locally in the dual Sobolev space $W^{-1,p'}_{\text{loc}}(\Omega)$, and $h^pT_k(u)^{1-p} \in W^{1,p}(\Omega)$ has compact support, the following manipulations are valid:
\begin{equation}\begin{split}\label{trivyoung}\int h^p \, d\nu_k & \leq \int h^p T_k(u)^{1-p}d\mu_k  =  \int \nabla{T_k(u)}^{p-2}\nabla T_k(u) \cdot \nabla\Bigl(\frac{h^p}{T_k(u)^{p-1}}\Bigl) dx\\
& \leq \Bigl(p\int \frac{h^{p-1}}{T_k(u)^{p-1}} \nabla{T_k(u)}^{p-2}\nabla T_k(u)\cdot \nabla h\\
&\;\;\;\;\;\;\; - (p-1) \int h^p \frac{\abs{\nabla T_k(u)}^p}{T_k(u)^p} dx \Bigl) \, \;\leq \int \abs{\nabla h}^p dx,
\end{split}\end{equation}
where we have used Young's inequality in the last line.  To prove the Lemma, we claim that:
\begin{equation}\label{rieszclaim}u^{p-1}\chi_{\{u<k\}} d\sigma \leq u^{p-1} d\nu_k \text{ on  } \text{supp}(h). \end{equation}  This will follow by an adaptation of a similar argument in \cite{DMM97}.  Indeed, since $T_{2k}(u)\in W^{1,p}_{\text{loc}}(\Omega)$, it follows that the set $\{u<k\}$ is quasi-open, see e.g. \cite{MZ1, DMM97}.  Therefore, there exists an increasing sequence $\phi_j \in W^{1,\infty}(\Omega)$, so that $\phi_j$ converges to  $\chi_{\{u<k\}}$ q.e..  This is a simple adaptation of the proof of Lemma 2.1 in \cite{DMG94}, since the functions $u_k$ considered in the proof of Lemma 2.1 of \cite{DMG94} can be chosen to be smooth.
It follows (see (\ref{widerclass})) that for any $\psi\in C^{\infty}_0(\text{supp}(h))$, that:
\begin{equation}\begin{split}\nonumber
\int_{\{u<k\}} \psi \phi_j u^{p-1}d\nu_k & = \int |\nabla T_k(u)|^{p-2}\nabla T_k(u) \cdot \nabla (\psi \phi_j) dx\\
& = \int |\nabla u|^{p-2}|\nabla u \cdot \nabla (\psi \phi_j) dx \geq \int \phi_j \psi u^{p-1} d\sigma,
\end{split}\end{equation}
the second equality here follows since $\phi_j$ is supported in $\{u<k\}$, and last inequality is by hypothesis.  Allowing $j\rightarrow \infty$, (\ref{rieszclaim}) follows.  Combining (\ref{rieszclaim}) with (\ref{trivyoung}) we conclude:
$$\int_{\{u<k\}}h^p d\sigma \leq \int |\nabla h|^p dx.
$$
Letting $k\rightarrow \infty$ with the aid of the monotone convergence theorem proves the lemma.
\end{proof}

It is easy to see by the definition of $p$-capacity that inequality (\ref{embineq}) implies the capacity inequality (\ref{capintro}) with constant $C=1$.  As was mentioned in the introduction, the converse is also true: if (\ref{capintro}) holds with constant $C = ((p-1)/p)^p$, then (\ref{embineq}) holds (see \cite{MazSob}).   Under the assumption that $\sigma \in L^{\infty}_{\text{loc}}$, (\ref{embineq}) is known to be equivalent to the existence of a solution to the inequality $\mathcal{L}(u) \geq 0$; see Theorem 2.3 in \cite{PT1}.  For more general $\sigma$ this relationship will be considered in \cite{JMV10}.

\subsection{}  Let us now prove Theorem \ref{gensigma}, we will do so by verifying an equivalent characterization of (\ref{capest}).

\begin{lem}\label{equivcap}  There is a constant $C$ so that (\ref{capest}) holds for all compact sets $E$ if and only if there is a constant $C_1>0$ so that:
\begin{equation}\label{equivcapst}
\int_E \mathbf{W}_{\alpha, s}(\chi_E d\sigma) \, d\sigma \leq C_1 \, \sigma(E)
\end{equation}
for all compact sets $E\subset \mathds{R}^n$.  Furthermore, if (\ref{equivcapst}) holds, then there is a positive constant $A>0$, depending on $\alpha, s$ and $n$, such that $$A^{-1}C_1 \leq C \leq AC_1.$$
\end{lem}
Lemma \ref{equivcap} is  well known, for instance a proof can be found in \cite{AH}, Theorem 7.2.1.

We will verify that the equivalent statement in Lemma \ref{equivcap} holds by first showing it holds for a dyadic analogue of the Wolff potential, and then using a standard shifting argument which goes back at least to Fefferman and Stein \cite{FS}; see also Garnett and  Jones \cite{GJ}.

To this end, we define the dyadic mesh at level $k$ for $k \in \mathds{Z}$, denoted by $\mathcal{D}_k$, as the collection of cubes in $\mathds{R}^n$ which are the translations by $2^k\lambda$ for $\lambda = (\lambda_1,...,\lambda_n) \in \mathbf{Z}^n$ of the cube $[0,2^k)^n$.  Then the dyadic lattice $\mathcal{D}$ is the collection of dyadic meshes $\mathcal{D}_k$, $k\in \mathds{Z}$.

With this notation, we define the discrete Wolff potentials $\mathcal{W}_{\alpha,s}^{ t}$ (see \cite{COV} for an in depth discussion) by
\begin{equation}\label{discwolff}
\mathcal{W}_{\alpha,s}^{\, t}(fd\sigma)(x) = \sum_{Q \in \mathcal{D}: \, x \in Q+t} c_Q \Bigl(\int_{Q+t} f (z) d\sigma(z)\Bigl)^{1/(s-1)}
\end{equation}
where $c_Q = \ell(Q)^{\frac{\alpha s-n}{s-1}}$ and $t\in \mathds{R}^n$.  Note that there is a constant $C$, depending only on $n, \alpha$ and $s$ (but not the shift $t$) so that for any nonnegative function $f$ 
\begin{equation}\label{dysmallcont}
\mathcal{W}_{\alpha,s}^{\,t} (f d\sigma) \leq C \mathbf{W}_{\alpha,s}(f d\sigma).
\end{equation}

We will use the following definition of the discrete Carleson measure.
\begin{defn}\label{dcarl}
Let $1<s<\infty$, and let $\sigma$ be a Borel measure on 
$\mathds{R}^n$. Then $\sigma$ is said to be a discrete Carleson measure if  
there is a positive constant $C=C(n, s)$ such that for each dyadic cube $P \in \mathcal{D}$ and  every  $t \in \mathds{R}^n$
 \begin{equation} \label{discarla}
 \sum_{Q\subset P, \; Q\in \mathcal{D}}c_Q \abs{Q+t}_{\sigma}^{s'} \leq C \, \abs{P+t }_{\sigma}.
\end{equation}
\end{defn}

\begin{rem}\label{dcarli}
It is well known that the inequality
 \begin{equation} \label{discarlb}
\sum_{Q \in \mathcal{D}} c_Q \left \vert \int_{Q+t} f d \sigma \right \vert^{s'} \le  C \, ||f||_{L^{s'}(d \sigma)}^{s'}
\end{equation}
holds for every $f \in L^{s'}(d \sigma)$ if and only if $\sigma$ is a discrete Carleson measure, and the constants in (\ref{discarla}) and 
 (\ref{discarlb}) are equivalent (see, e.g.,  \cite{NTV99, COV}). From this 
 it is immediate that if $\sigma$ is a Carleson measure then 
 $\chi_E \, d \sigma$ is also a Carleson measure, for every measurable $E\subset \mathds{R}^n$. 
\end{rem}

We now formulate a discrete analogue of the characterization in Lemma \ref{equivcap} which will be sufficient for our purposes, where 
we make use of Definition~\ref{dcarl} and Remark~\ref{dcarli}. 

\begin{lem} \label{carllem} Suppose there is a positive solution $u$ to the integral inequality (\ref{genintineq}).  Then the measure $\sigma$ is a discrete Carleson measure, that is there is a positive constant $C=C(n, s, C_0)$ such that for each dyadic cube $P \in \mathcal{D}$ and  every compact set $E\subset \mathds{R}^n$,
 \begin{equation} \label{discarl}\sum_{\substack{Q \subset P\\Q\in \mathcal{D}}}c_Q \abs{(Q+t)\cap E}_{\sigma}^{s'} \leq C\abs{(P+t) \cap E}_{\sigma}.
\end{equation}
Furthermore, we have that 
\begin{equation} \label{discarl2}\sum_{Q
\in \mathcal{D}}c_Q \abs{(Q+t)\cap E}_{\sigma}^{s'} \leq C\abs{E} _{\sigma}.
\end{equation}
\end{lem}

\begin{proof}  We will prove (\ref{discarl}). The proof of (\ref{discarl2}) follows by the same reasoning.  The proof is rather reminiscent of the classical Schur's Lemma.  First note that by hypothesis and (\ref{dysmallcont}) there is a positive function $u$ together with a constant $C>0$ so that
$$
u(x) \geq C\mathcal{W}_{\alpha,s}^{t}(u^{s-1}d\sigma)(x)
$$
and hence, using H\"{o}lder's inequality, we see that:
\begin{equation}\nonumber
\begin{split}
\sum_{\substack{Q \subset P\\Q\in \mathcal{D}}}c_Q &\abs{(Q+t)\cap E}_{\sigma}^{s'} = \sum_{Q \subset P}c_Q \Bigl\{\int_{(Q+t) \cap E}
u^{-\frac{s-1}{s}} \cdot u^{\frac{s-1}{s}}d\sigma\Bigr\}^{s'}\\
& \leq \sum_{Q \subset P}c_Q \int_{(Q+t)\cap E}
u^{-1} \,d\sigma \cdot \Bigl\{\int_{(Q+t)\cap E)} u^{s-1}d\sigma\Bigr\}^{\frac{1}{s-1}}. 
\end{split}
\end{equation}
By interchanging summation and integration, which is permitted by the monotone convergence theorem, we see that the last line is equal to:
\begin{equation}\begin{split}\nonumber
\int_{(P+t) \cap E}& u^{-1}\sum_{Q \subset P}c_Q\Big\{\int_{(Q+t)\cap E} u^{s-1} d\sigma\Bigr\}^{\frac{1}{s-1}}\chi_{Q+t}(x)d\sigma\\
&\leq\int_{(P+t) \cap E} u^{-1} \cdot  \mathcal{W}_{\alpha,s}^{\,t} (u^{s-1}d\sigma) d\sigma\\
& \leq C \int_{(P+t) \cap E} u^{-1}\cdot u  \,d\sigma = C \abs{(P+t) \cap E}_{\sigma}. \end{split}\end{equation}\end{proof}

We now state a suitable version of the dyadic  averaging result which will be sufficient for our purposes.

\begin{lem}\label{dyshift} There is a positive integer $j_0 \in \mathds{N}$ so that for any $j \in \mathds{Z}$ there is a constant $C = C(n,\alpha, s)$, not depending on $j$, so that
\begin{equation}\nonumber
\mathbf{W}_{\alpha,s}^{2^j}(fd\sigma)(x) \leq C\dashint_{B(0,2^{j+j_0})} \mathcal{W}_{\alpha,s}^{\, t} (fd\sigma)(x) \,  d t
\end{equation}
where $\mathbf{W}_{\alpha,s}^{2^j}$ is the local Wolff potential defined in (\ref{localwolff}).
\end{lem}
A proof of this lemma can be found, for instance, in \cite{COV}.  

We will next use the dyadic shifting argument to prove the following lemma:
\begin{lem}\label{mWolffest} Suppose $u$ is a positive solution of (\ref{genintineq}) with constant $C_0$.  Then there is a constant $C=C(n,\alpha, s)$ so that for any compact set $E \subset \mathds{R}^n$, and each $m \in \mathds{N}$ the measure $\sigma$ satisfies:
\begin{equation}\nonumber
\int_{E}\Bigl(\mathbf{W}_{\alpha,s}(\chi_{E}d\sigma)\Bigr)^m \,  d\sigma \leq C^m \, m! \,  \sigma(E). 
\end{equation}
\end{lem}

\begin{rem}  This Lemma in the case $m=1$ shows that Lemma \ref{equivcap} is satisfied, and hence proves Theorem \ref{gensigma}.  We prove the Lemma in the form stated as it gives us an exponential integrability result, which will be very useful in the sequel (see Corollary \ref{corexpest} below).
\end{rem}

\begin{proof}  Let $E$ be a compact set.  Then first we note that by Fatou's lemma,  
\begin{equation}\nonumber \int_{E}\Bigl(\mathbf{W}_{\alpha,s}(\chi_{E}d\sigma)\Bigr)^m \,  d\sigma \leq \liminf_{k \rightarrow \infty}\int_{E}\Bigl(\mathbf{W}_{\alpha,s}^{ 2^k}(\chi_{E}d\sigma)\Bigr)^m \, d\sigma 
\end{equation}
where $\displaystyle \mathbf{W}_{\alpha,s}^{ 2^k}(\chi_Ed\sigma)(x) = \int_0^{2^k}\Bigl(\frac{\sigma(B(x,r)\cap E)}{r^{n-\alpha s}}\Bigl)^{1/(s-1)}\frac{dr}{r}$. 

It therefore suffices to find a bound on the right hand side of the preceding inequality which is independent of $k$.  Lemma \ref{dyshift} yields:
\begin{equation}\nonumber\begin{split}
\int_{E}\Bigl(\mathbf{W}_{\alpha,s}^{ 2^k}&(\chi_{E}d\sigma)\Bigr)^m \, d\sigma\\
&\leq C^m\int_{E}\Bigl(  \dashint_{B(0,2^{k+j_0})} \mathcal{W}_{\alpha,s}^{t }(\chi_E d\sigma) d t  \Bigr)^m \, d\sigma\\
& \leq C^m\Bigl(\dashint_{B(0,2^{k+j_0})}\Bigl( \int_{E}\Bigl( \mathcal{W}_{\alpha,s}^{\,t }(\chi_{E}d\sigma)\Bigr)^m \, d\sigma  \Bigr)^{\frac{1}{m}}dt\Bigr)^m , 
\end{split}
\end{equation}
where the second inequality follows from Minkowski's integral inequality.  

We will need the elementary summation by parts inequality:
\begin{equation}\label{sumparts1}\Bigl(\sum_{j=1}^{\infty} \lambda_j\Bigl)^m \leq m \sum_{j=1}^{\infty}\lambda_j \Bigl(\sum_{k=1}^j \lambda_k \Bigl)^{m-1}
\end{equation}
which holds for any nonnegative sequence $\{\lambda_j\}_j$ and $m\geq 1$.  We apply Lemma \ref{carllem} to the dyadic Wolff potential, after an $m$ fold application of  (\ref{sumparts1}).  Indeed, considering  the inner integral in the right hand side of the last line above, we obtain:
\begin{equation}\begin{split}\label{mfoldest}
\int_{E}\Bigl(& \mathcal{W}_{\alpha,s}^{\, t}(\chi_{E}d\sigma)\Bigr)^m d\sigma\\
& = \int_{E}\Bigl( \sum_{Q\in \mathcal{D}}c_Q \abs{Q+t\cap E}_{\sigma}^{\frac{1}{s-1}}\chi_{Q+t}\Bigr)^m \, d\sigma\\
& \leq m! \int_{E}\sum_{Q_1\in \mathcal{D}}c_{Q_1} \abs{Q_1+t\cap E}_{\sigma}^{\frac{1}{s-1}} 
\ldots \sum_{Q_m\subset Q_{m-1}}c_{Q_m} \abs{Q_m+t\cap E}_{\sigma}^{\frac{1}{s-1}}\chi_{Q_m+t} \, d\sigma\\
&=m!\sum_{Q_1\in \mathcal{D}}c_{Q_1} \abs{Q_1+t\cap E}_{\sigma}^{\frac{1}{s-1}} \ldots \sum_{Q_m\subset Q_{m-1}}c_{Q_m} \abs{Q_m+t\cap E}_{\sigma}^{\frac{s}{s-1}}\\
&\leq m! C^m \sigma(E). 
\end{split}
\end{equation}
In the last line we have used (\ref{discarl}) $m-1$ times and then (\ref{discarl2}) once.  Bringing together our estimates proves the lemma.
\end{proof}

The following exponential integrability result easily follows from Lemma \ref{mWolffest}, the power series representation of the exponential, and the monotone convergence theorem.
\begin{cor} \label{corexpest} Suppose $u$ is a positive solution of (\ref{genintineq}).  If we let $\beta>0$ so that $C\beta <1$, where $C$ is the constant appearing in Lemma \ref{mWolffest}, then we have the following:
\begin{equation}\label{expest}\int_{E}e^{\;\beta \mathbf{W}_{\alpha,s}(\chi_{E}d\sigma)(y)} d\sigma(y) \leq \frac{1}{1-C\beta} \sigma(E)
\end{equation}
whenever $E$ is a compact set.
\end{cor}

In our next result, we specialize  (\ref{capest})  to when the set $E$ is a ball.  By a standard formula for the capacity of a ball (see \cite{AH}, Chapter 5), 
\begin{equation}\label{potest} \sigma (B(x,r)) \leq C_1 \, \text{cap}_{\alpha,s}(B(x,r)) = 
C_2 r^{n-\alpha s}
\end{equation}
for all balls $B(x,r)$, where  $C_2 = A \,  C_1$, and  $A$ depends only on $n, \alpha$ and $s$.  However, as is well known, (\ref{potest}) does not imply  (\ref{capest}) for all compact sets $E$. 

Our next lemma shows that the tail of the Wolff potential is nearly constant, which is a key estimate to our construction of the supersolution.

\begin{lem}\label{lemtailest}  Let $\sigma$ be a Borel measure satisfying (\ref{potest}). Then there is a positive constant $C=C(n, \alpha, s, C_2)>0$, so that for all $x \in \mathbb{R}^n$ and $y \in B(x,t)$,  $t>0$, it follows: 
\begin{equation}\label{tailest} \abs{\int_t^{\infty} \left [ \Bigl(\frac{\sigma(B(x,r))}{r^{n-\alpha s}}\Bigr)^{\frac{1}{s-1}} - \Bigl(\frac{\sigma(B(y,r))}{r^{n-\alpha s}}\Bigr)^{\frac{1}{s-1}} \right ]\frac{dr}{r}} \leq C .
\end{equation}
\end{lem}

The proof of Lemma \ref{lemtailest} is a modification of an argument due to Frazier and Verbitsky, \cite{FV1} for integral operators with kernels satisfying a quasimetric condition.  The extension to the nonlinear potential is elementary, but also technical and rather lengthy.   Due to this we present the proof elsewhere, in Appendix A of \cite{JV}.

\section{Lower bounds for nonlinear integral equations, \\ the proof of Theorems \ref{l1lowbd} and \ref{l2lowbd}} \label{lowbds}

\subsection{}In this section, we will prove Theorems \ref{l1lowbd} and \ref{l2lowbd}.   Recall the operator 
$$\mathcal{N}(f)(x) = \mathbf{W}_{\alpha, s}(f^{s-1}d\sigma)(x).$$
We will begin this section by proving a lower bound for solutions of the inequality:
\begin{equation}\label{genfund}
u(x) \geq C_0 \mathcal{N}(u)(x) + C_0 \abs{x-x_0}^{\frac{\alpha s - n}{s-1}}.
\end{equation}
We will show the following theorem:
\begin{thm}\label{genlowbd} Suppose that $u$ satisfies (\ref{genfund}) with constant $C_0$.  Then there is a constant $c=c(n,\alpha,s, C_0)>0$ such that:
\begin{equation}\begin{split}\label{genlowbdst}
u(x) \geq c \abs{x-x_0}^{\frac{\alpha s -n}{s-1}}& \exp\Bigl(c \int_0^{\abs{x-x_0}}\Bigl(\frac{\sigma(B(x,r))}{r^{n-\alpha s}}\Bigl)^{1/(s-1)}\frac{dr}{r}\Bigl)\\
&\cdot \exp\Bigl(c\int_0^{\abs{x-x_0}}\frac{\sigma(B(x_0,r))}{r^{n-\alpha s}}\frac{dr}{r}\Bigl).
\end{split}\end{equation}
\end{thm}

Theorems \ref{l1lowbd} and \ref{l2lowbd} will follow quickly from this theorem, as we shall show once it is proved. 

We shall prove Theorem \ref{genlowbd} by iterating (\ref{genfund}).  To illustrate the iteration, suppose that $T$ is a homogeneous superlinear operator acting on nonnegative functions, i.e. that $T(cf) = cT(f)$ for $c>0$ and $T(f+g) \geq T(f) + T(g)$ whenever $f$ and $g$ are nonnegative measurable functions.   In addition suppose that $u$ satisfies the inequality: 
\begin{equation}\label{genintineq1}
u \geq T(u) +f
\end{equation}
where $f\geq 0$.  Now we define the $j$-th iterate of $T$ by $T^j(f) = T(T^{j-1}(f))$, for all $j\geq 2$.  Iterating (\ref{genintineq1}) $m$ times yields:
\begin{equation}\begin{split}\nonumber
u & \geq T(T(\dots T(T(u) + f)+f \cdots)+f)+f\\
& \geq T^m(f) + T^{m-1}(f) + \dots + T(f) + f,
\end{split}
\end{equation}
and since $m$ here was arbitrary,
$$ u \geq \sum_{j=1}^{\infty} T^j(f) + f.$$
Now, if $1<s\leq 2$, it is clear from Minkowski's inequality that $\mathcal{N}$ is a superlinear homogeneous operator and hence if $u$ is a solution of (\ref{genfund}), then:
$$ u \geq \sum_{j=1}^{\infty} C_0^j \mathcal{N}^j(\abs{\cdot-x_0}^{\frac{\alpha s-n}{s-1}}) + C_0\abs{x-x_0}^{\frac{\alpha s-n}{s-1}}.$$
However, if $2<s<n$, the operator $\mathcal{N}$ does not fall within this framework.  In this case we consider an operator $\mathcal{T}(f) = \mathcal{N}(f^{1/(s-1)})^{s-1} = (\mathbf{W}_{\alpha, s}(f))^{s-1}$.  Then by Minkowski's inequality, $\mathcal{T}$ is superlinear, and it is homogenous, and so we may apply the above discussion.   If $u$ satisfies (\ref{genfund}), then we have that:
$$ u^{s-1}(x) \geq C\mathcal{T}^j(u^{s-1})(x) + C\abs{x-x_0}^{\alpha s-n}
$$
where $C$ is a positive constant depending on $n, \alpha, s$ and $C_0$.  Hence, we see that 
$$ u^{s-1}(x) \geq \sum_{j=1}^{\infty}C^j\mathcal{T}^j(\abs{\cdot - x_0}^{\alpha s-n})(x) + C\abs{x-x_0}^{\alpha s-n}.$$ 
By comparing iterates of $\mathcal{T}$ with the iterates of $\mathcal{N}$, we obtain 
$$ u(x) \geq \Bigl(\sum_{j=1}^{\infty}C^j\mathcal{N}^{j}(\abs{\cdot - x_0}^{\frac{\alpha s-n}{s-1}})(x)^{s-1}\Bigl)^{1/(s-1)} + C\abs{x-x_0}^{\frac{\alpha s-n}{s-1}}.$$
Thus, by Jensen's (or H\"{o}lder's) inequality, we have that for any $q>1$, 
\begin{equation}\nonumber
u \geq C \sum_{j=1}^{\infty} j^{(q\frac{2-s}{s-1})}C^j \mathcal{N}_1^j(\abs{\cdot - x_0}^{\frac{\alpha s-n}{s-1}})(x) + C\abs{x-x_0}^{\frac{\alpha s-n}{s-1}}
\end{equation}
where $C$ is a positive constant depending on $q, n, s, \alpha$ and $C_0$. 

We summarize this discussion as follows:

\begin{lem}\label{lowbdsum}  Suppose $u$ is a solution of (\ref{genfund}) with constant $C_0$. Then there is a constant $C=C(n,s, \alpha, C_0)>0$ so that if $1<s \leq 2$, it follows:
\begin{equation}\label{Nlowbdsleq2} u \geq \sum_{j=1}^{\infty} C^j \mathcal{N}^j(\abs{\cdot-x_0}^{\frac{\alpha s-n}{s-1}}) + C\abs{x-x_0}^{\frac{\alpha s-n}{s-1}} .\end{equation}
If $2<s<n$, then for any $q>1$, 
\begin{equation}\label{Nlowbdsgeq2}
u \geq C(q) \sum_{j=1}^{\infty} j^{(q\frac{2-s}{s-1})}C^j \mathcal{N}_1^j(\abs{\cdot - x_0}^{\frac{\alpha s-n}{s-1}})(x) + C\abs{x-x_0}^{\frac{\alpha s-n}{s-1}}
\end{equation}
where $C(s) =C(q,n,\alpha,s,C_0)>0$.
\end{lem}

\subsection{Proof of Theorem \ref{genlowbd}} 

Suppose that $u$ is a solution of (\ref{genfund}). Then clearly $u$ also satisfies (\ref{genintineq}), and hence by Theorem \ref{gensigma}, (\ref{capest}) holds for all balls compact sets $E$.  Hence there is a constant $C(\sigma)>0$ so that:
$$ C(\sigma) = \sup_{E}\frac{\sigma(E)}{\text{cap}_{\alpha ,s}(E)}<\infty .$$
where the supremum is taken over compact sets $E$ so that $\text{cap}_{\alpha, s}(E)>0$.  Note that this implies $\sigma (B(x,r)) \leq A C(\sigma) r^{n-\alpha s}$ for all balls $B(x,r)$, where $A$ is a positive constant depending on $n, \alpha$ and $s$.
To prove Theorem \ref{genlowbd}, we estimate the iterates $\mathcal{N}^j(\abs{\cdot - x_0}^{\frac{\alpha s-n}{s-1}})$.  We will do this in two lemmas, giving us two bounds.  We then average the two bounds to conclude the theorem.

\begin{lem}\label{greenlowbd1} For a given $x \in \mathds{R}^n$, define $j_x$ to be the integer so that
\begin{equation}\nonumber
2^{j_x}\leq \abs{x-x_0} < 2^{j_x +1} . 
\end{equation}
Then, with $B_k = B(x_0, 2^k)$, for any $m\geq 1$, 
\begin{equation}\begin{split}\label{lowlemfund1}
\mathcal{N}^m (\abs{\cdot - x_0}^{\frac{\alpha s - n}{s-1}})(x) & \geq \Bigl( \frac{s-1}{n-\alpha s} 8^{\frac{\alpha s-n}{s-1}}\Bigl)^m \abs{x - x_0}^{\frac{\alpha s-n}{s-1}}\\
& \cdot \Bigl(\frac{1}{m!}\Big\{\sum_{k = -\infty}^{j_x} 2^{k(\alpha s-n)} \sigma (B_{k+1}\backslash B_k)\Bigl\}^m\Bigl)^{1/(s-1)} .
\end{split}
\end{equation}
\end{lem}

\begin{proof}  We will prove this lemma by induction. Let us  recall the definition of the operator $\mathcal{N}$:
\begin{equation}\nonumber
\mathcal{N} (\abs{\cdot - x_0}^{\frac{\alpha s-n}{s-1}})(x) = \int_0^{\infty} \Bigl( \frac{1}{r^{n-\alpha s}} \int_{B(x,r)} \abs{y-x_0}^{\alpha s-n} d\sigma(y) \Bigl)^{1/(s-1)} \frac{dr}{r} .
\end{equation}
First, restrict the integration in the variable $r$ to $r>4\abs{x-x_0}$.  Then, observe that as $r>4\abs{x-x_0}$:  $B(x_0, 2\abs{x-x_0}) \subset B(x,r)$.   This results in the bound:
\begin{equation}\begin{split}\label{lowfundstep1}
\mathcal{N} (\abs{\cdot - x_0}^{\frac{\alpha s-n}{s-1}})(x) \geq & \int_{4\abs{x-x_0}}^{\infty} r^{\frac{\alpha s-n}{s-1}} \frac{dr}{r} \\& \cdot\Bigl(\int_{B(x_0, 2\abs{x-x_0})} \abs{y-x_0}^{\alpha s-n} d\sigma(y) \Bigl)^{1/(s-1)} . 
\end{split}\end{equation}
Now, recalling the definition of $j_x$, we have:
\begin{equation}\nonumber
\int_{B(x_0, 2\abs{x-x_0})} \abs{y-x_0}^{\alpha s-n} d\sigma(y) \geq \sum_{k=-\infty}^{j_x} 2^{(k+1)(\alpha s-n)} \sigma(B_{k+1}\backslash B_k) . 
\end{equation}
Using this and evaluating the integral in (\ref{lowfundstep1}) yields the case where $k=1$. 

Now suppose (\ref{lowlemfund1}) holds for some $m$.  Then by the induction hypothesis, and the observation above:
\begin{equation}\begin{split}\nonumber
\mathcal{N}^{m+1}&(\abs{\cdot - x_0}^{\frac{\alpha s-n}{s-1}})(x) \geq \Bigl(\frac{s-1}{n-\alpha s} 8^{\frac{\alpha s-n}{s-1}}\Bigl)^m\frac{s-1}{n-\alpha s}4^{\frac{\alpha s-n}{s-1}}\abs{x - x_0}^{\frac{\alpha s-n}{s-1}}\\
& \cdot\Bigl( \frac{1}{m!}\int_{B(x_0, 2\abs{x-x_0})}|z-x_0|^{\alpha s-n}\Bigl(\sum_{\ell =-\infty}^{j_y}2^{\ell(\alpha s-n)}\sigma(B_{\ell+1}\backslash B_{\ell})\Bigl)^m d\sigma(y)\Bigl)^{1/(s-1)} . 
\end{split}\end{equation}
We now consider the integral 
\begin{equation}\label{llfint}
 \int_{B(x_0, 2\abs{x-x_0})}|z-x_0|^{\alpha s -n}\Bigl(\sum_{\ell =-\infty}^{j_y}2^{\ell(\alpha s-n)}\sigma(B_{\ell+1}\backslash B_{\ell})\Bigl)^m d\sigma(y) . 
\end{equation}
To complete the inductive step and hence prove the lemma it suffices to show that (\ref{llfint}) is greater than 
\begin{equation}\label{llfint2}
\frac{2^{\alpha s-n}}{m+1} \Bigl(\sum_{\ell =-\infty}^{j_x}2^{\ell(\alpha s-n)}\sigma(B_{\ell+1}\backslash B_{\ell})\Bigl)^{m+1} .
\end{equation}
To this end, note that by the definition of $j_x$, (\ref{llfint}) is greater than 
\begin{equation}\label{llfint3}
\sum_{k = -\infty}^{j_x} 2^{(k+1)(\alpha s-n)} \int_{B_{k+1}\backslash B_k}\Bigl(\sum_{\ell =-\infty}^{j_y}2^{\ell(\alpha s-n)}\sigma(B_{\ell+1}\backslash B_{\ell})\Bigl)^m d\sigma(y) . 
\end{equation}
We next remark that for all $y\in B_{k+1}\backslash B_k$, we have by definition $j_y = k$, and so (\ref{llfint3}) equals:
\begin{equation}\label{llfintstep2}
2^{\alpha s-n}\sum_{k = -\infty}^{j_x} 2^{k(\alpha s-n)} \sigma(B_{k+1}\backslash B_k)\Bigl(\sum_{\ell =-\infty}^{k}2^{\ell(\alpha s-n)}\sigma(B_{\ell+1}\backslash B_{\ell})\Bigl)^m  . 
\end{equation}
But an application of the elementary summation by parts inequality (\ref{sumparts1}) now gives that (\ref{llfintstep2}) is greater than (\ref{llfint2}).  This concludes the proof of the Lemma.
\end{proof}

By using Jensen's (or H\"{o}lder's) inequality, inserting Lemma \ref{greenlowbd1} into the bounds (\ref{Nlowbdsleq2}) and (\ref{Nlowbdsgeq2}) in Lemma \ref{lowbdsum} yields the existence of positive constants $c_1$ and $c_2$, depending on $n, \alpha, s$ and $C_0$, so that:
\begin{equation}\nonumber
u(x) \geq c_1 \abs{x-x_0}^{\frac{\alpha s-n}{s-1}}\exp\Bigl(c_2 \sum_{\ell =-\infty}^{j_x}2^{\ell(\alpha s-n)}\sigma(B_{\ell+1}\backslash B_{\ell}) \Bigl) . 
\end{equation}
But, since $\sigma$ satisfies (\ref{potest}), we may further estimate the sum. Indeed, 
\begin{equation}\nonumber
 \sum_{\ell =-\infty}^{j_x}2^{\ell(\alpha s-n)}\sigma(B_{\ell+1}\backslash B_{\ell}) \geq C \int_{0}^{\abs{x-x_0}}\frac{\sigma(B(x,r))}{r^{n-\alpha s}}\frac{dr}{r},
\end{equation}
where $C = C(n,\alpha ,s)>0$.  Hence we may conclude that there are positive constants $c_1$ and $c_2$, depending on $n, \alpha, s, C_0$ and $C(\sigma)$,  so that:
\begin{equation}\nonumber
u(x) \geq c_1 \abs{x-x_0}^{\frac{\alpha s-n}{s-1}}\exp\Bigl(c_2 \int_{0}^{\abs{x-x_0}}\frac{\sigma(B(x,r))}{r^{n-\alpha s}}\frac{dr}{r} \Bigl) . 
\end{equation}
The second part of the exponential build up in Theorem \ref{genlowbd} is accounted for in the following lemma:
\begin{lem}\label{greenlowbd2} For any $m\geq 1$,
\begin{equation}\begin{split}\label{lowlemfund2}
\mathcal{N}^m(\abs{\cdot - x_0}^{\frac{\alpha s-n}{s-1}})(x) \geq & (3/2)^{\frac{\alpha s-n}{s-1}}\frac{1}{m!} \abs{x-x_0}^{\frac{\alpha s-n}{s-1}} \\ & \cdot \Bigl(\int_0^{\abs{x-x_0}}\Bigl(\frac{\sigma(B(x,r/2))}{r^{n-\alpha s}}\Bigl)^{1/(s-1)}\frac{dr}{r}\Bigl)^{m} . 
\end{split}\end{equation}
\end{lem}

\begin{proof} We will prove Lemma \ref{greenlowbd2} when $m=3$, as the case of general $m$ is completely similar.  The proof is based on the following claim:

For any locally finite Borel measures $\sigma$ and $\omega$, and $x, x_0 \in \mathbf{R}^n$:
\begin{equation}\begin{split}\label{lowfundclaim}
\int_{0}^{\abs{x-x_0}} & \Bigl(\frac{1}{r^{n-\alpha s}}\int_{B(x, r/2)} \Bigl\{\int_r^{\infty} \Bigl( \frac{1}{u^{n-\alpha s}} \omega(B(y,u))\Bigl)^{1/(s-1)}\frac{du}{u} \Bigl\}^{s-1} d\sigma(y)\Bigl)^{1/(s-1)} \frac{dr}{r}\\
& \geq \int_{0}^{\abs{x-x_0}} \Bigl(\frac{\sigma(B(x, r/2)}{r^{n-\alpha s}}\Bigl)^{1/(s-1)} \int_{r}^{\infty} \Bigl( \frac{1}{u^{n-\alpha s}} \omega(B(x,u/2))\Bigl)^{1/(s-1)}\frac{du}{u} \frac{dr}{r}.
\end{split}\end{equation}

The claim is just the triangle inequality.  Suppose that $\abs{y-x} < r/2$ and $r<u$, then whenever $z\in B(x,u/2)$:  $B(x, u/2) \subset B(y, u)$.  Thus, $\omega(B(x,u/2)) \leq \omega (B(y,u))$.  The claim (\ref{lowfundclaim}) then follows by using this estimate in the left hand side and noting that the inner integrand no longer depends on $y$.

The Lemma will follow from repeated use of the claim.  First, by using definition and restricting domains of integration:
\begin{equation}\begin{split}\label{lflint3}
\mathcal{N}^3(\abs{\cdot - x_0}^{\frac{\alpha s-n}{s-1}})(x) & \geq  \int_{0}^{\abs{x-x_0}} \Bigl(\frac{1}{r^{\alpha s-n}}\int_{B(x, r/2)} \\
& \cdot \Bigl\{\int_r^{\infty}\Bigl( \frac{1}{u^{n-\alpha s}} \omega(B(y,u))\Bigl)^{1/(s-1)}\frac{du}{u} \Bigl\}^{s-1} d\sigma(y)\Bigl)^{1/(s-1)} \frac{dr}{r}\
\end{split}\end{equation}
where:
\begin{equation}\nonumber
\omega(B(y,u)) = \int_{B(y,u)} \Bigl\{\int_{0}^{\infty} \Bigl(\frac{1}{t^{n-\alpha s}} \int_{B(z,t)} \abs{w-x_0}^{\alpha s-n} d\sigma(w)\Bigl)^{1/(s-1)} \frac{dt}{t} \Bigl\}^{s-1}d\sigma(z) . 
\end{equation}
Applying the claim (\ref{lowfundclaim}) to (\ref{lflint3}), we have that (\ref{lflint3}) is greater than:
\begin{equation}\nonumber
 \int_{0}^{\abs{x-x_0}} \Bigl(\frac{\sigma(B(x, r/2)}{r^{n-\alpha s}}\Bigl)^{1/(s-1)} \int_{r}^{\infty} \Bigl( \frac{1}{u^{n-\alpha s}} \omega(B(x,u/2))\Bigl)^{1/(s-1)}\frac{du}{u} \frac{dr}{r} . 
\end{equation}
Let's now consider the integral:
\begin{equation}\nonumber
 \int_{r}^{\infty} \Bigl( \frac{1}{s^{n-\alpha s}} \omega(B(x,u/2))\Bigl)^{1/(s-1)}\frac{du}{u} \geq  \int_{r}^{\abs{x-x_0}} \Bigl( \frac{1}{u^{n-\alpha s}} \omega(B(x,u/2))\Bigl)^{1/(s-1)}\frac{du}{u} .
\end{equation}
Then we may rewrite the right hand side of this last line as:
\begin{equation}\begin{split}\label{lflint4}
\int_{r}^{\abs{x-x_0}} & \Bigl( \frac{1}{u^{n-\alpha s}}  \int_{B(x,u/2)} \Bigl\{\int_{0}^{\infty} \Bigl(\frac{1}{t^{n-\alpha s}} \mu(B(z,t))\Bigl)^{1/(s-1)} \frac{dt}{t} \Bigl\}^{s-1}d\sigma(z)\Bigl)^{1/(s-1)}\frac{du}{u} 
\end{split}\end{equation}
where
\begin{equation}\nonumber
\mu(B(z,t)) =  \int_{B(z,t)} \abs{w-x_0}^{\alpha s-n} d\sigma(w) . 
\end{equation}
Now, restricting the integral over $t$ to $t>u$, and applying the claim (\ref{lowfundclaim}) with $\omega = \mu$, we see that (\ref{lflint4}) is greater than 
\begin{equation}\nonumber
\int_{r}^{\abs{x-x_0}} \Bigl( \frac{1}{u^{n-\alpha s}}  \sigma (B(x,u/2))\Bigl)^{1/(s-1)} \int_u^{\abs{x-x_0}} \Bigl(\frac{1}{t^{n-\alpha s}} \mu(B(x,t/2))\Bigl)^{1/(s-1)} \frac{dt}{t}\frac{du}{u}
\end{equation}
where we have also restricted the integration over $t$ to $t< \abs{x-x_0}$.  Now, let us consider:
\begin{equation}\begin{split}\nonumber
 \int_u^{\abs{x-x_0}} & \Bigl(\frac{1}{t^{n-\alpha s}} \mu(B(x,t))\Bigl)^{1/(s-1)} \frac{dt}{t} \\ & = \int_u^{\abs{x-x_0}} \Bigl(\frac{1}{t^{n-\alpha s}}  \int_{B(x,t/2)} \abs{w-x_0}^{\alpha s-n} d\sigma(w)\Bigl)^{1/(s-1)} \frac{dt}{t} . 
\end{split}\end{equation}
But, for $w\in B(x,t/2)$, note that: $\abs{w-x_0} < 3/2 \abs{x-x_0}$.  Thus, 
\begin{equation}\nonumber\begin{split}
 \int_u^{\abs{x-x_0}} \Bigl(&\frac{1}{t^{n-\alpha s}} \mu(B(x,t/2))\Bigl)^{1/(s-1)} \frac{dt}{t} \\
 & \geq (3/2)^{\frac{\alpha s-n}{s-1}}\abs{x-x_0}^{\frac{\alpha s-n}{s-1}} \int_u^{\abs{x-x_0}} \Bigl(\frac{1}{t^{n-\alpha s}} \sigma(B(x,t/2))\Bigl)^{1/(s-1)} \frac{dt}{t}  . 
\end{split}\end{equation}
Putting together what we have so far,
\begin{equation}\nonumber\begin{split}
\mathcal{N}^3( & \abs{\cdot - x_0}^{\frac{\alpha s-n}{s-1}})(x)  \geq (3/2)^{\frac{\alpha s-n}{s-1}}\abs{x-x_0}^{\frac{\alpha s-n}{s-1}}\int_{r=0}^{\abs{x-x_0}}\Bigl(\frac{\sigma(B(x,r/2))}{r^{n-\alpha s}}\Bigl)^{1/(s-1)}\\
& \cdot\int_{u=r}^{\abs{x-x_0}}\Bigl(\frac{\sigma(B(x,u/2))}{u^{n-\alpha s}}\Bigl)^{1/(s-1)}\int_{t=u}^{\abs{x-x_0}}\Bigl(\frac{\sigma(B(x,t/2))}{t^{n-\alpha s}}\Bigl)^{1/(s-1)}\frac{dt}{t}\frac{du}{u}\frac{dr}{r} .
\end{split}
\end{equation}
Integration by parts now yields the Lemma in the case $m=3$.  It is easy to see that a completely similar argument works for general $m$, using the claim (\ref{lowfundclaim}) $m-1$ times as we have done  twice in the above argument.  Thus the Lemma is proved.
\end{proof}

As with Lemma \ref{greenlowbd1}, we readily see that applying Lemma \ref{greenlowbd2} to the iterates in the bounds (\ref{Nlowbdsleq2}) and (\ref{Nlowbdsgeq2}) of Lemma \ref{lowbdsum} yields the existence of positive constants $c_1$ and $c_2$, depending on $n, \alpha, s$ and $C_0$, so that 
\begin{equation}\label{lowbdrem1}
u(x) \geq c_1 \abs{x-x_0}^{\frac{\alpha s-n}{s-1}} \exp \Bigl(c_2\int_{0}^{\abs{x-x_0}} \Bigl(\frac{\sigma(B(x,r/2))}{r^{n-\alpha s}}\Bigl)^{1/(s-1)} \frac{dr}{r}\Bigl) . 
\end{equation}

But, since $C(\sigma)<\infty$, we can replace $\sigma(B(x,r/2))$ by $\sigma(B(x,r))$  in the integral in (\ref{lowbdrem1}). Indeed, by change of variables:
\begin{equation}\nonumber
\int_0^{\abs{x-x_0}} \Bigl(\frac{\sigma(B(x,r/2))}{r^{n-\alpha s}}\Bigl)^{1/(s-1)} \frac{dr}{r} = 2^{\frac{\alpha s-n}{s-1}}\int_0^{\abs{x-x_0}/2}\Bigl(\frac{\sigma(B(x,r))}{r^{n-\alpha s}}\Bigl)^{1/(s-1)} \frac{dr}{r}  , 
\end{equation}
and by (\ref{potest}):
\begin{equation}\nonumber
\int_{\abs{x-x_0}/2}^{\abs{x-x_0}} \Bigl(\frac{\sigma(B(x,r))}{r^{n-\alpha s}}\Bigl)^{1/(s-1)} \frac{dr}{r} \leq C(n,\alpha, s,C(\sigma)) . 
\end{equation}
Thus we conclude that there are positive constants $c_1$ and $c_2$ depending on $n, \alpha, s, C(\sigma)$ and $C_0$ so that 
\begin{equation}\nonumber
u(x) \geq c_1 \abs{x-x_0}^{\frac{\alpha s-n}{s-1}} \exp \Bigl(c_2\int_{0}^{\abs{x-x_0}} \Bigl(\frac{\sigma(B(x,r))}{r^{n-\alpha s}}\Bigl)^{1/(s-1)} \frac{dr}{r}\Bigl) . 
\end{equation}

\begin{proof} [Proof of Theorem \ref{genlowbd}]
We have showed that if $u$ is a solution of (\ref{genfund}) with constant $C_0$, then there are constants $c_1$ and $c_2$, depending on $n, \alpha, s, C_0$ and $C(\sigma)$, so that the following two inequalities hold:
\begin{equation}\label{pflowbdineq1}
u(x) \geq c_1 \abs{x-x_0}^{\frac{\alpha s-n}{s-1}}\exp\Bigl(c_2 \int_{0}^{\abs{x-x_0}}\frac{\sigma(B(x,r))}{r^{n-\alpha s}}\frac{dr}{r} \Bigl) , 
\end{equation}
\begin{equation}\label{pflowbdineq2}
u(x) \geq c_1 \abs{x-x_0}^{\frac{\alpha s-n}{s-1}} \exp \Bigl(c_2\int_{0}^{\abs{x-x_0}} \Bigl(\frac{\sigma(B(x,r))}{r^{n-\alpha s}}\Bigl)^{1/(s-1)} \frac{dr}{r}\Bigl) . 
\end{equation}
Averaging (\ref{pflowbdineq1}) and (\ref{pflowbdineq2}) with the inequality of the arithmetic mean and geometric mean, $a/2+ b/2 \geq \sqrt{ab}$, yields the required lower bound for solutions of (\ref{genfund}), and hence completes the proof of Theorem \ref{genlowbd}.
\end{proof}

\begin{proof}[Proof of Theorems \ref{l1lowbd} and \ref{l2lowbd}]
The capacity estimates have been proven in Remark \ref{caprem} so it remains to prove the bounds on the fundamental solutions.  Suppose first that $u$ is a fundamental solution of $\mathcal{L}$.  Then, as a result of the Wolff potential estimate, $u$ satisfies the inequality (\ref{N1ineq}), which is (\ref{genfund}) in the case when $\alpha=1$ and $s=p$.  Applying Theorem \ref{genlowbd} when specialized to this case is precisely the bound (\ref{l1lowest}) of Theorem \ref{l1lowbd}. 

Similarly, if $u$ is a fundamental solution of $\mathcal{G}$, then $u$ satisfies (\ref{N2ineq}), which is just (\ref{genfund}) when $\alpha = \frac{2k}{k+1}$ and $s=k+1$ and so we may apply Theorem \ref{genlowbd}.  We again see that the bound (\ref{genlowbdst}) in Theorem \ref{genlowbd} with this choice of $\alpha$ and $s$ is exactly the required bound (\ref{l2lowest}) in Theorem \ref{l2lowbd}.
\end{proof}

\section{Construction of a supersolution}\label{construct}
\subsection{}In this section we will construct a solution corresponding to the integral inequality (\ref{fundcons}) below, which as we have already seen is closely related to the fundamental solutions of $\mathcal{L}$ and $\mathcal{G}$.  Suppose that $v$ is a solution of the integral inequality:
\begin{equation}\label{fundcons}
v(x) \geq C_0 \mathcal{N}(v)(x) + \abs{x-x_0}^{\frac{\alpha s-n}{s-1}}
\end{equation}
where
$$\mathcal{N}(f)(x) = \mathbf{W}_{\alpha, s}(f^{s-1}d\sigma)(x)
$$
for any positive constant $C_0>0$.   Then by Theorem \ref{gensigma} there is a constant $C(\sigma)>0$ such that $\sigma$ satisfies:
\begin{equation}\label{sigmacons} \sigma(E) \leq C(\sigma) \text{cap}_{\alpha, s}(E)
\end{equation}
for all compact sets $E\subset \mathds{R}^n$.  By Corollary \ref{corexpest}, a consequence of this is that there is a positive constant $A=A(s,\alpha, n)$ so that:
\begin{equation}\label{expsigmacons}
\int_{B(x,r)}e^{\; \beta \mathbf{W}_{\alpha, s}(\chi_{B(x,r)}d\sigma)} \, d\sigma \leq \frac{1}{1-\beta A C(\sigma)}\sigma(B(x,r)), 
\end{equation}
provided $\beta A C(\sigma)<1$.  In addition note that by standard capacity estimates we may also assume that 
$$AC(\sigma) \geq \sup_{x\in \mathds{R}^n, \, r>0} \frac{\sigma(B(x,r))}{r^{n-\alpha s}}
$$
and hence the hypothesis of Lemma \ref{lemtailest} are satisfied. 

To solve the inequality (\ref{fundcons}) it suffices to find a function $u$ so that $v \geq \abs{x-x_0}^{\frac{\alpha s-n}{s-1}}$ and $v\geq C \mathcal{N}(v)$.  With this in mind the following theorem will be enough for our purposes.  Recall that $B_k = B(x_0, 2^k)$ and $j_x$ is defined to be the integer so that $2^{j_x} \leq \abs{x-x_0} < 2^{j_x+1}$.

\begin{thm} \label{superfund}  Let $\sigma$ be a measure satisfying (\ref{sigmacons}) (and hence (\ref{expsigmacons})).  In addition suppose that 
\begin{equation}\label{localfinitecons}
\int_0^1\frac{\sigma(B(x_0,r))}{r^{n-\alpha s}}\frac{dr}{r} < \infty.
\end{equation}
Define a function $v$ by the following:
\begin{equation}\begin{split}
v(x) = \abs{x-x_0}^{\frac{\alpha s-n}{s-1}}& \exp\Bigl(\beta \sum_{\ell = -\infty}^{j_x} 2^{\ell(\alpha s-n)}\sigma(B_{\ell+1}) \Bigl)\\
& \cdot \exp \Bigl(\beta \int_{0}^{\abs{x-x_0}}\Bigl(\frac{\sigma(B(x,r)}{r^{n-\alpha s}}\Bigl)^{1/(s-1)}\frac{dr}{r}\Bigl) . 
\end{split} \end{equation}
Then, if $C(\sigma)$ is sufficiently small, there exists a positive $\beta = \beta (C(\sigma), n, \alpha, s)$, along with a positive constant $C_0=C_0(\beta, n, \alpha, s, \sigma)$ so that 
$$v \geq C_0 \mathcal{N}(v), \text{  and in addition } \inf_{x\in \mathds{R}^n}v(x) = 0.$$
\end{thm}

\begin{rem} The condition (\ref{localfinitecons}) is only used to ensure that $v$ is not identically infinite. By inspection of the bound in Theorem \ref{genlowbd} it is clear that if it is not satisfied then any fundamental solution is identically infinite.
\end{rem}

\begin{proof}
We let $\mathcal{N}(v)= I + II$, where $I$ is defined by 
\begin{equation}
I = \int_{0}^{\abs{x-x_0}/2} \Bigl(\frac{1}{r^{n-\alpha s}}\int_{B(x,r)} v^{s-1}(y) \, d\sigma(y)\Bigl)^{1/(s-1)}\frac{dr}{r} . 
\end{equation}
First note that for any $y\in B(x,r)$ with $r \leq\abs{x-x_0}/2$, we have that $\abs{y-x_0} \leq (3/2)\abs{x-x_0}$ and $j_y \leq j_x+1$.  In addition note that for such $y$, 
$$\abs{y-x_0} \geq \abs{x-x_0} - \abs{x-y} \geq \abs{x-x_0}/2 .$$
These two observations, when plugged into $I$, yield:
\begin{equation}\begin{split}\nonumber
I \leq 2^{\frac{n-\alpha s}{s-1}}& \abs{x-x_0}^{\frac{\alpha s-n}{s-1}} \exp\Bigl(\beta \sum_{\ell = -\infty}^{j_x + 1} 2^{\ell(\alpha s-n)} \sigma(B_{\ell+1})\Bigl) \int_0 ^{\frac{1}{2}\abs{x-x_0}}  \Bigl(\frac{1}{r^{n-\alpha s}} \\
&\cdot  \int_{B(x,r)}\exp \Bigl((s-1)\beta \int_{0}^{\frac{3}{2}\abs{x-x_0}} \Bigl( \frac{\sigma(B(y,t))}{t^{n- \alpha s} }\Bigl)^{\frac{1}{s-1}}\frac{dt}{t}\Bigl) \, d\sigma(y)\Bigl)^{\frac{1}{s-1}}\frac{dr}{r} .
\end{split} 
\end{equation}
We now pay attention to the integral 
\begin{equation}\label{thmstep1}
\int_{B(x,r)} \exp \Bigl((s-1)\beta \int_{0}^{\frac{3}{2}\abs{x-x_0}} (t^{\alpha s-n} \sigma(B(y,t))^{1/(s-1)}\frac{dt}{t}\Bigl) d\sigma(y) . 
\end{equation}
Note that we may rewrite (\ref{thmstep1}) as 
\begin{equation}\begin{split}\label{supfund1}
& \int _{B(x,r)} \exp \Bigl((s-1)\beta \int_0 ^{r} \Bigl(\frac{\sigma(B(y,t)}{t^{n-\alpha s}}\Bigl)^{1/(s-1)}\frac{dt}{t}\Bigl) \\
& \cdot \exp \Bigl((s-1)\beta \int_{r}^{\frac{3}{2}\abs{x-x_0}} \left [ \Bigl(\frac{\sigma(B(y,t)}{t^{n-\alpha s}}\Bigl)^{1/(s-1)} - \, \Bigl(\frac{\sigma(B(x,t)}{t^{n-\alpha s}}\Bigl)^{1/(s-1)}\right]\frac{dt}{t}\Bigl)\; d\sigma(y)\\
& \cdot \exp \Bigl((s-1)\beta \int_{r} ^{\frac{3}{2}\abs{x-x_0}} \Bigl(\frac{\sigma(B(x,t)}{t^{n-\alpha s}}\Bigl)^{1/(s-1)}\frac{dt}{t}\Bigl) . 
\end{split}\end{equation}
By the Wolff potential tail estimate, Lemma \ref{lemtailest}, it follows:
\begin{equation}\nonumber
\abs{\int_{r}^{\frac{3}{2}\abs{x-x_0}} \left [ \Bigl(\frac{\sigma(B(y,t)}{t^{n-\alpha s}}\Bigl)^{1/(s-1)} -  \Bigl(\frac{\sigma(B(x,t)}{t^{n-\alpha s}}\Bigl)^{1/(s-1)}\right]\frac{dt}{t}} \leq C(n, \alpha, s,C(\sigma)) . 
\end{equation}
Thus (\ref{supfund1}) is less than a constant multiple of:
\begin{equation}\begin{split}\label{preexp}
\int_{B(x,r)} & \exp \Bigl((s-1)\beta \int_0 ^{r} \Bigl(\frac{\sigma(B(y,t)}{t^{n-\alpha s}}\Bigl)^{1/(s-1)}\frac{dt}{t}\Bigl) \, d\sigma(y) \\&\cdot \exp \Bigl((s-1)\beta \int_{r} ^{\frac{3}{2}\abs{x-x_0}} \Bigl(\frac{\sigma(B(x,t)}{t^{n-\alpha s}}\Bigl)^{1/(s-1)}\frac{dt}{t}\Bigl) . 
\end{split}
\end{equation}
Now, provided $\beta C(\sigma)$ is small enough we may apply (\ref{expsigmacons}), and hence we may estimate the integral in (\ref{preexp}) by:
\begin{equation}\begin{split}\nonumber
\int_{B(x,r)} & \exp \Bigl((s-1)\beta \int_0 ^{r} \Bigl(\frac{\sigma(B(y,t)}{t^{n-\alpha s}}\Bigl)^{1/(s-1)}\frac{dt}{t}\Bigl) \, d\sigma(y)\\
& \leq \int_{B(x, 2r)} \exp \Bigl( (p-1)\beta W_{\alpha,s}^{\sigma} (\chi_{B(x,2r)})(y) \Bigl) \, d\sigma(y) \leq C \sigma(B(x, 2r)) . 
\end{split}\end{equation}
Putting these estimates together, there is a constant $C=C(n, \alpha, s, C(\sigma))$ so that:
\begin{equation}\begin{split}\nonumber
I \leq C &  \abs{x-x_0}^{\frac{\alpha s-n}{s-1}}\exp\Bigl(\beta \sum_{\ell = -\infty}^{j_x + 1} 2^{\ell(\alpha s-n)} \sigma(B_{\ell+1})\Bigl)\\
& \cdot  \int_0^{\abs{x-x_0}}\Bigl(\frac{\sigma(B(x,2r))}{r^{n-\alpha s}}\Bigl)^{1/(p-1)} \cdot \exp \Bigl(\beta \int_{r} ^{\frac{3}{2}\abs{x-x_0}} \Bigl(\frac{\sigma(B(x,t)}{t^{n-\alpha s}}\Bigl)^{1/(s-1)}\frac{dt}{t}\Bigl)\, \frac{dr}{r} . 
\end{split}\end{equation}
But now note since $\sigma$ satisfies (\ref{sigmacons}), we have, for any $\rho>0$:
\begin{equation}\label{annest}
 \int_{\rho} ^{2\rho} \Bigl(\frac{\sigma(B(x,t)}{t^{n-\alpha s}}\Bigl)^{1/(s-1)}\frac{dt}{t} \leq C, \;\;\text{and}\;\;
2^{(j_x+1)(\alpha s-n)} \sigma(B_{j_x+2})\leq C,
\end{equation}
where in this last display the constant depends on $n, \alpha, s$ and $C(\sigma)$, but is independent of $\rho$.
By a change of variables and (\ref{annest}), we see there is a positive constant $C=C(n, \alpha, s, C(\sigma))$, so that:
\begin{equation}\begin{split}\nonumber
I \leq C &  \abs{x-x_0}^{\frac{\alpha s-n}{s-1}}\exp\Bigl(\beta \sum_{\ell = -\infty}^{j_x} 2^{\ell(\alpha s-n)} \sigma(B_{\ell+1})\Bigl)\\
& \cdot \int_0^{\abs{x-x_0}}\Bigl(\frac{\sigma(B(x,r))}{r^{n-\alpha s}}\Bigl)^{1/(s-1)} \cdot \exp \Bigl(\beta \int_{r} ^{\abs{x-x_0}} \Bigl(\frac{\sigma(B(x,t)}{t^{n-\alpha s}}\Bigl)^{1/(s-1)}\frac{dt}{t}\Bigl) \, \frac{dr}{r} . 
\end{split}\end{equation}
An application of integration by parts now yields $I \leq C \, v$ for a positive constant $C=C(n, \alpha, s, C(\sigma))$.

We next consider the remainder of the Wolff potential $II$. By writing the integral as a sum over dyadic annuli, it is not difficult to see that there exists a constant $C>0$, depending on $n, s$ and $\alpha$, so that:
\begin{equation}
II \leq C \sum_{k = j_x}^{\infty}2^{k\frac{\alpha s-n}{s-1}}\Bigl(\int_{B(x,2^k)} v^{s-1} \,  d\sigma \Bigl)^{1/(s-1)} . 
\end{equation}
Let us first consider a single integral in the sum.  Since $k\geq j_x$, 
it follows that $B(x,2^k) \subset B(x_0, 2^{k+2})$.  Thus, 
\begin{equation}\begin{split}\label{supfund2}
\int_{B(x,2^k)} v^{s-1} \, d\sigma & \leq \int_{B(x_0, 2^{k+2})} v^{s-1} 
\,  d\sigma  = \sum_{\ell = -\infty}^{k+2} \int_{B_{\ell} \backslash B_{\ell-1}} v^{s-1} \, d\sigma . 
\end{split}\end{equation}
We now concentrate on one term in the sum on the right hand side of (\ref{supfund2}). Observe that for $z\in B_{\ell}\backslash B_{\ell-1}$, we have $2^{\ell} \geq \abs{z-x_0} \geq 2^{\ell - 1}$ and $j_z = \ell-1$.  This yields:
\begin{equation}\begin{split}\nonumber
 \int_{B_{\ell} \backslash B_{\ell-1}} v^{s-1}(z) d\sigma(z) \leq & 2^{(\ell - 1)(p-n)} \exp \Bigl( \beta (s-1)\sum_{m=-\infty}^{\ell-1} 2^{m(\alpha s-n)}\sigma(B_{m+1}) \Bigl)\\
 & \cdot \int_{B_\ell}\exp\Bigl((s-1)\beta \int_{0}^{2^{\ell}}\Bigl(\frac{\sigma(B(y,t))}{t^{n-\alpha s}}\Bigl)^{1/(s-1)} \frac{dt}{t}\Bigl) \, d\sigma(y). 
\end{split}\end{equation}
But, again, if we suppose that $\beta C(\sigma)$ is small, then by the exponential integration result (\ref{expsigmacons}), there is a constant $C=C(n,p,s,C(\sigma))>0$ so that:
\begin{equation}\nonumber
 \int_{B_\ell}\exp\Bigl((s-1)\beta \int_{0}^{2^{\ell}}\Bigl(\frac{\sigma(B(y,t))}{t^{n-\alpha s}}\Bigl)^{1/(s-1)} \frac{dt}{t}\Bigl) \,  d\sigma(y)\leq C\sigma(B(x,2^{\ell+1})) . 
 \end{equation}
Thus, plugging this into (\ref{supfund2}), we find that there is a constant $C = C(n,p,s,C(\sigma))>0$ so that:
\begin{equation}\begin{split}\label{supfund3}
\int_{B(x,2^k)} v^{s-1} d\sigma(z) \leq  C& \sum_{\ell = -\infty}^{k+2}2^{\ell(\alpha s-n)} \sigma(B(x,2^{\ell+1}))\\& \cdot \exp \Bigl( \beta (s-1)\sum_{m=-\infty}^{\ell-1} 2^{m(\alpha s-n)}\sigma(B_{m+1}) \Bigl) . 
\end{split}\end{equation}
Next, consider the following summation by parts estimate (see \cite{FV2}).  Suppose that $\{\lambda_j\}_j$ is a nonnegative sequence such that $0\leq \lambda_j \leq 1$.  Then:
\begin{equation}\label{sumparts2}\sum_{j=0}^{\infty}\lambda_j e^{ \sum_{k=j}^{\infty}\lambda_k} \leq 2\, e^{\sum_{j=0}^{\infty}\lambda_j} . 
\end{equation}

Provided $C(\sigma)\leq 1$, we may apply (\ref{sumparts2}) to see that the right hand side of (\ref{supfund3}) is less than a constant (depending on $n, \alpha, s, C(\sigma)$) multiple of:
\begin{equation}\nonumber
\exp\Bigl((s-1)\beta \sum_{\ell = -\infty}^{k+2} 2^{\ell(\alpha s-n)}\sigma(B_{\ell+1})\Bigl) . 
\end{equation}
Hence (as we may bound two top terms in the above sum using the $C(\sigma)$ condition), 
\begin{equation}\label{decaystep}
II \leq C \sum_{k =j_x}^{\infty} 2^{k\frac{\alpha s-n}{s-1}}\exp\Bigl(\beta \sum_{\ell = -\infty}^{k} 2^{\ell(\alpha s-n)} \sigma(B_{\ell+1})\Bigl) . 
\end{equation}
This is less than a constant multiple of $u$ provided $C(\sigma)$ is small enough.  Indeed, note that the right hand side of (\ref{decaystep}) is a constant multiple of:
\begin{equation}\begin{split}\label{decaystep2}
2^{j_x \frac{\alpha s-n}{s-1}} & \exp\Bigl(\beta \sum_{\ell = -\infty}^{j_x} 2^{\ell (\alpha s-n)} \sigma(B_{\ell +1}) \Bigl)\cdot\sum_{k=0}^{\infty}2^{k\frac{\alpha s-n}{s-1}}\exp(\beta \sum_{\ell = 1}^{k} 2^{\ell (\alpha s-n)} \sigma(B_{\ell +1})\Bigl) . 
\end{split}\end{equation} 
Now, using the definitions of $j_x$, $v$ and also (\ref{sigmacons}), it is immediate that (\ref{decaystep2}) is less than 
\begin{equation}\label{decaystep3}
C \,  v(x)  \sum_{k=0}^{\infty}2^{k\frac{\alpha s-n}{s-1}}\exp\Bigl(\beta A C(\sigma)^{s-1} k\Bigl)
\end{equation}
where $C=C(n,\alpha, s, C(\sigma))$ and $A=A(n,s, \alpha)$.  Now, with $C(\sigma)$ small enough, this series converges and so $v\geq C II$ for a positive constant $C>0$ depending on $n, s, \alpha, C(\sigma)$.

It is left to see that $\inf_{x\in \mathds{R}^n}v(x) = 0$.  To this end, first note that we can chose $C(\sigma)$ sufficiently small so that:
\begin{equation}\begin{split}\abs{x-x_0}^{\frac{\alpha s-n}{s-1}}& \exp\Bigl(\beta \sum_{\ell = -\infty}^{j_x} 2^{\ell(\alpha s-n)}\sigma(B_{\ell+1}) \Bigl)\\
& \cdot \exp \Bigl(\beta \int_{1}^{\abs{x-x_0}}\Bigl(\frac{\sigma(B(x,r)}{r^{n-\alpha s}}\Bigl)^{1/(s-1)}\frac{dr}{r}\Bigl) \;\; \rightarrow 0,\, \text{ as  } \abs{x}\rightarrow \infty.
\end{split}\end{equation}
Indeed, this follows from the argument in (\ref{decaystep3}), using the condition (\ref{sigmacons}), along with noting that:
$$ \sum_{\ell = -\infty}^{1} 2^{\ell(\alpha s-n)}\sigma(B_{\ell+1})  \leq C \int_0^1\frac{\sigma(B(x_0,r))}{r^{n-\alpha s}}\frac{dr}{r} < \infty .$$
Let us define a sequence $a_j$ by:
$$a_j = \inf_{x\in B(0, 2^j) \backslash B(0, 2^{j-1})} \int_0 ^1 \Bigl(\frac{\sigma(B(x,r))}{r^{n - \alpha s}}\Bigl)^{1/(s-1)}\frac{dr}{r}.$$
To finish the proof it therefore suffices to show that $a_j$ tends to zero as $j \rightarrow \infty$.  First suppose $s\geq 2$, then consider:
\begin{equation}\nonumber
b_R = \frac{1}{R^n}\int_{B(0,R)} \int_0^1\Bigl(\frac{\sigma(B(x,r))}{r^{n-\alpha s}}\Bigl)^{1/(s-1)}\frac{dr}{r} dx .
\end{equation}
By Fubini and H\"{o}lder's inequality,
\begin{equation} \label{estbj} b_R \leq C \int_0^1 \frac{1}{r^{\frac{n-\alpha s}{s-1}+1}}\Bigl(\frac{1}{R^n}\int_{B(0,R)} \sigma(B(x,r)) dx\Bigl)^{1/(s-1)} dr. 
\end{equation}
Then by Fubini once again, $\displaystyle \int_{B(0,R)} \sigma(B(x,r)) dx \leq C r^n \sigma(B(0,2R)) \leq Cr^n R^{n-p}$,
where we have used (\ref{sigmacons}) in this last line.  Plugging this estimate into (\ref{estbj}) we find that $b_R \rightarrow 0$ as $R\rightarrow \infty$.  This clearly implies that $a_j$ is a null sequence, since $a_j \leq C b_{2^j}$ for a positive constant independent of $j$.

Now let $1<s<2$ and note that for any integer $k$: 
\begin{equation}\label{wolfflessriesz}\begin{split}
\bigl( \int_0^{2^k}\Bigl( \frac{\sigma(B(x,r))}{r^{n-\alpha s}}\bigl)^{1/(s-1)}& \frac{dr}{r}\Bigl)^{s-1}  \leq C \Bigl(\sum_{j=-\infty}^{k} \Bigl(\frac{\sigma(B(x, 2^j))}{2^{j(n - \alpha s )}}\Bigl)^{1/(s-1)}\Bigl)^{s-1}\\
&\leq C\sum_{j=-\infty}^{k}\frac{\sigma(B(x, 2^j))}{2^{j(n-\alpha s )}} \leq C  \int_0^{2^k}\frac{\sigma(B(x,r))}{r^{n-\alpha s}}\frac{dr}{r}.
\end{split}\end{equation} 
Since the previous argument shows that:
$$ \frac{1}{R^n}\int_{B(0,R)} \int_0^1\frac{\sigma(B(x_0,r))}{r^{n-\alpha s}}\frac{dr}{r} dx \;\; \rightarrow 0, \,\text{ as }\; R\rightarrow \infty,
$$
we conclude that $a_j \rightarrow 0$ as $j\rightarrow \infty$ when $1<s<2$.  Thus $\inf_{x\in \mathds{R}^n} v(x) = 0$.
\end{proof}

\section{Proofs of Theorems \ref{l1upbd} and \ref{l2upbd}}\label{existence}

\subsection{} In this section we will prove Theorems \ref{l1upbd} and \ref{l2upbd}.  We make use of the construction in Section \ref{construct}.  Combined with a simple iteration scheme based on weak continuity, which is similar to those in \cite{PV, PV2}.  Let us first consider the quasilinear case. 

\begin{proof}[Proof of Theorem \ref{l1upbd}] 
Recall that we denote by $C(\sigma)$ the positive (and by assumption finite) constant:
$$C(\sigma) = \sup_{E}\frac{\sigma(E)}{\text{cap}_{1,p}(E)} , 
$$
where the supremum is taken over all compact sets $E \subset \mathds{R}^n$ of positive capacity.  Note that by Lemma \ref{quascapequiv};
$$C(\sigma) \geq C \sup_{E}\frac{\sigma(E)}{\text{cap}_{p}(E)} $$
where $\text{cap}_p$ is the $p$-capacity, and $C=C(n,p)>0$.
Suppose first that:
\begin{equation}\label{quaslocfin}
\int_0^1 \frac{\sigma(B(x_0,r))}{r^{n-p}}\frac{dr}{r} = \infty . 
\end{equation}
Then, we see that by Theorem \ref{l1lowbd} any fundamental solution $u(x,x_0) \equiv \infty$, and there is nothing to prove.  Hence we may assume that  the integral in (\ref{quaslocfin}) is finite, and so we may apply Theorem \ref{superfund}. This implies that if $C(\sigma)$ is sufficiently small, in terms of $n$ and $p$, then there is a function $v \in L^{p-1}_{\text{loc}}(\sigma)$ and a constant $C_0 >0$, depending on $n$ and $p$ such that:
\begin{equation}\label{quasintsuper}
v(x) \geq C_0 \mathbf{W}_{1,p}^{\sigma}(v^{p-1})(x) + \tilde{K}\abs{x-x_0}^{\frac{p-n}{p-1}},
\end{equation}
and $\inf_{x\in \mathds{R}^n}v(x) = 0$.  Here $\tilde{K} = \frac{p-1}{n-p} K$, with $K=K(n,p)>0$ the same constant that appears in the Wolff potential estimate, Theorem \ref{kmpotest}.  Indeed, recalling that $j_x$ is the integer such that $2^{j_x}\leq \abs{x-x_0} \leq 2^{j_x+1}$, we can let 
\begin{equation}\begin{split}\nonumber
v(x) = 2\tilde{K} \abs{x-x_0}^{\frac{p-n}{p-1}}& \exp\Bigl(\beta \sum_{\ell = -\infty}^{j_x} 2^{\ell(p-n)}\sigma(B(x_0, 2^{\ell+1})) \Bigl)\\
& \cdot \exp \Bigl(\beta \int_{r=0}^{\abs{x-x_0}}\Bigl(\frac{\sigma(B(x,r))}{r^{n-p}}\Bigl)^{1/(p-1)}\frac{dr}{r}\Bigl)
\end{split}\end{equation}
for a suitable choice of $\beta = \beta(n,p)>0$.
Let $u_0 = G( \, \cdot, x_0)$ where $G(x,x_0)$ is defined in (\ref{unperturbed}).  Then $u_0$ is $p$-superharmonic in $\mathds{R}^n$ and $-\Delta_p u_0(x) = \delta_{x_0}$ (in fact $u_0$ is the unique such solution, see, e.g. \cite{KicVer1}). By choice of $\tilde{K}$ (assuming $K>1$) we have that $u_0 \leq v$, and hence $u_0 \in L^{p-1}_{\text{loc}}(\sigma)$.  Let $\epsilon>0$ be such that $\epsilon K \leq C_0$, then we claim that there exists a sequence $\{u_m\}_{m\geq 0}$ of functions which are $p$-superharmonic in $\mathds{R}^n$, $u_m \in L^{p-1}_{\text{loc}}(\sigma)$:
\begin{equation}\label{pfundapprox}-\Delta_p u_m = \epsilon \sigma (u_{m-1})^{p-1} + \delta_{x_0} ,\;\; \text{and} \inf_{x\in \mathds{R}^n}u_m(x) = 0, 
\end{equation}
and in addition $u_m(x) \leq v(x)$.  The existence of this sequence can be shown by the techniques of \cite{PV2}, using the notion of \textit{renormalized solutions}.  However, as we are dealing exclusively with $p$-superharmonic functions this detour would be somewhat artificial and so we prove the claim directly.  Indeed, suppose that $u_1, \dots, u_{m-1}$ have been constructed. Then, $\epsilon \sigma (u_{m-1})^{p-1} + \delta_{x_0}$ is a locally finite Borel measure.  For each $j \in \mathds{N}$, let $u_m^j$ be a positive $p$-superharmonic function such that 
$$-\Delta_p u_m^j = \epsilon \sigma (u_{m-1})^{p-1} \chi_{B(x_0, 2^j)} + \delta_{x_0}\; \text{ in } \; \mathds{R}^n . 
$$
The existence of such a $p$-superharmonic function is guaranteed by \cite{Kil1}, Theorem 2.10.  
By subtracting a positive constant, we may assume that $\inf_{x\in \mathds{R}^n} u_{m}^j = 0$.

Now, by the global Wolff potential estimate and since $\mathbf{W}_{1,p}(\delta_{x_0}) = \frac{p-1}{n-p}\abs{x-x_0}^{\frac{p-n}{p-1}}$, we find that
$$u_m^j(x) \leq K\epsilon \mathbf{W}_{1,p}^{\sigma}u^{p-1}_{m-1}(x) + \tilde{K} \abs{x-x_0}^{\frac{p-n}{p-1}} . 
$$
But since $u_{m-1} \leq v$,
$$u_m^j(x) \leq K\epsilon \mathbf{W}_{1,p}^{\sigma}v^{p-1}(x) + \tilde{K} \abs{x-x_0}^{\frac{p-n}{p-1}} . 
$$
By choice of $\epsilon>0$ so that $K\epsilon \leq C_0$, we conclude that $u_{m}^j(x) \leq v(x)$.

Appealing now to Theorem \ref{limsuper} (\cite{KM}, Theorem 1.17), we find a subsequence $u_m^{j_k}$ and an $p$-superharmonic function $u_m$ such that $u_m^{j_k}(x) \rightarrow u_m(x)$ for almost every $x\in \mathds{R}^n$.  Thus $u_m(x) \leq v(x)$ and hence $\inf_{x\in \mathds{R}^n} u_m(x) =0$.  The claim is then completed by appealing to Theorem \ref{weakcont} to see that 
$$-\Delta_p u_m = \epsilon \sigma (u_{m-1})^{p-1} + \delta_{x_0} \;\;\; \text{in}\;\;\mathds{R}^n . 
$$

Now, since $u_m(x) \leq v(x)$, for all $m$, we may again find a subsequence $\{u_{m_k}\}_k$ and a positive $p$-superharmonic function $u$ so that $u_{m_k}(x) \rightarrow u(x)$ almost everywhere.  Since it follows that $u(x)\leq v(x)$, we have that $\inf_{x\in \mathds{R}}u(x)=0$.  Finally, by Theorem \ref{weakcont}, we may conclude that:
$$-\Delta_p u = \epsilon \sigma u^{p-1} + \delta_{x_0} \;\;\; \text{in}\; \; \mathds{R}^n . 
$$
This completes the proof of Theorem \ref{l1upbd} with the potential $\tilde{\sigma} = \epsilon \sigma$, once we notice that:
$$ \sum_{\ell = -\infty}^{j_x} 2^{\ell(p-n)}\sigma(B(x_0, 2^{\ell+1}))\leq C \int_0^{\abs{x-x_0}} \frac{\sigma(B(x_0,r))}{r^{n-p}}\frac{dr}{r} + C
$$
for a positive constant $C$ depending on $n$ and $p$.
\end{proof}

\subsection{} For the Hessian existence theorem, we may state the following Lemma, contained in \cite{PV2}, Lemma 4.7.

\begin{lem}\cite{PV2}\label{Hessiansequence} Let $\mu$ and $\nu$ be nonnegative locally finite Borel measures in $\mathds{R}^n$, so that $\mu \leq \nu$ and $\mathbf{W}_{\frac{2k}{k+1}, k+1}\nu<\infty$ almost everywhere.  Suppose that $u  \geq 0$  satisfies $-u \in \Phi^k(\mathds{R}^n)$, $\mu_k[-u]=\mu$, and $u$ is a pointwise a.e. limit of a subsequence of the sequence $\{u_m\}_m$, with $-u_m \in \Phi^k(B(x_0, 2^{m+1}))$ and 
\begin{equation}\begin{cases}\nonumber
\;\mu_k[-u_m] = \mu \chi_{B(x_0,2^m)} \;\;\; \text{in}\;\;B(x_0, 2^{m+1})\\
\; u_m = 0\;\; \text{on} \;\;\; \partial B(x_0, 2^{m+1}) . 
\end{cases}\end{equation}
Then there is a nonnegative function so that $-w \in \Phi^k(\mathds{R}^n)$, $w\geq u$, and
\begin{equation}\nonumber
\mu_k[-w] = \nu \;\;\;\; \text{and}\;\; \inf_{x\in \mathds{R}^n}v(x) =0 . 
\end{equation}
Moreover, $w$ is a pointwise a.e. limit of a sequence $\{w_m\}_m$, so that $-w_m \in \Phi^k(B(x_0, 2^{m+1}))$ and 
\begin{equation}\begin{cases}\nonumber
\;\mu_k[-w	_m] = \nu \chi_{B(x_0,2^m)} \;\;\; \text{in}\;\;B(x_0, 2^{m+1})\\
\; w_m = 0\;\; \text{on} \;\;\; \partial B(x_0, 2^{m+1}) . 
\end{cases}\end{equation}
\end{lem}

\begin{proof}[Proof of Theorem \ref{l2upbd}]  This is very similar to the previous proof so we will be slightly brief to avoid repetition.  As in the previous proof, the theorem is only nontrivial in the case when,
$$\int_0^1 \frac{\sigma(B(x_0,r))}{r^{n-2k}}\frac{dr}{r} < \infty. 
$$
Hence if $C(\sigma)$ small enough, where now 
$$C(\sigma) = \sup_{E \, \text{compact}}\frac{\sigma(E)}{\text{cap}_{2k/(k+1), k+1}(E)} , 
$$
then we may apply Theorem \ref{superfund} to find a positive function $v$ such that $\inf_{x\in \mathds{R}^n}v(x) =0$ and 
$$v(x) \geq C_0 \mathbf{W}^{\sigma}_{\frac{2k}{k+1}, k+1}(v^k)(x) + \tilde{K} \abs{x-x_0}^{2/k - n}
$$
with $\tilde{K} = \frac{k}{n-2k} K$.  Here $K$ is a constant appearing in the global Wolff potential bound Theorem \ref{Hesswolff}.  

Let $\epsilon>0$ be such that $\epsilon K \leq C_0$. Let $u_0 = c(n,k)\abs{x-x_0}^{2/k-n}$, where $c(n,k) = (\frac{k}{n-2k})\cdot({n\choose k}\omega_{n-1})^{-1/k}$.  Then $u_0$ is the (unique) fundamental solution of the $k$-Hessian operator in $\mathds{R}^n$, see \cite{TWHess3}.  By a repeated application of Lemma \ref{Hessiansequence}, we find a sequence $\{u_m\}_m$ of nonnegative functions so that $-u_m \in \Phi^k(\mathds{R}^n)$, $\inf_{x\in \mathds{R}^n}u_m(x) =0$, $u_m \in L^k_{\text{loc}}(\sigma)$ and 
$$\mu_k[-u_m] = \epsilon \sigma (u_{m-1})^{p-1} + \delta_{x_0} . 
$$
Furthermore, as in the previous proof, we see that by choice of $\tilde{K}$ that $u_m \leq v$.  Now, appealing to the weak continuity of the $k$-Hessian operator (Theorem \ref{Hessweakcont}), we assert the existence of a nonnegative $u$ such that $-u \in \Phi^k (\mathds{R}^n)$,
$$\mu_k[-u] = \epsilon \sigma u^{k} + \delta_{x_0},
$$
and $u\leq v$.  Hence $\inf_{x\in \mathds{R}^n}u(x) = 0$.  Thus, noting Lemma \ref{hesscapequiv}, we see that Theorem \ref{l2upbd} is proved with potential $\tilde{\sigma} = \epsilon \sigma$, once we make the easy observations that $v$ is comparable to the right hand side of the bound (\ref{l2upest}).
\end{proof}

\subsection{Criteria for equivalence of perturbed and unperturbed fundamental solutions} \label{isolated} In this short section we conclude the paper with necessary and sufficient conditions for fundamental solutions of $\mathcal{L}$, defined by (\ref{defnl1}), to be equivalent to the fundamental solutions of the $p$-Laplacian.  Similar results also holds for the $k$-Hessian operator.  Recall the fundamental solution of $-\Delta_p$, which we denoted by $G(x,x_0)$ in (\ref{unperturbed}), and the Wolff and Riesz potentials from (\ref{localwolff}) and (\ref{localriesz}) respectively.
\begin{cor}  \label{fundsolnequiv}  Suppose that there is a positive constant $c>0$ such that for all $x_0 \in \mathds{R}^n$ (\ref{fundsolnequivineq}) holds whenever $u(x,x_0)$ is a fundamental solution of $\mathcal{L}$.  Then $\sigma(E) \leq \rm{cap}_p(E)$ for all compact sets $E$, and furthermore (\ref{item1fund}) and (\ref{item2fund}) hold.

Conversely, suppose that (\ref{item1fund}) holds if $1<p\leq 2$, or (\ref{item2fund}) holds if $p\geq 2$.  Then there exists a positive constant $C$, depending on $n$ and $p$, such that if $\sigma(E) \leq C \rm{cap}_p(E)$ for all compact sets $E$, then for any $x_0 \in \mathds{R}^n$ there is a fundamental solution $u(x,x_0)$ of $\mathcal{L}$ with pole at $x_0$ satisfying (\ref{fundsolnequivineq}) for a constant $c = c(n,p)>0$.
\end{cor}

The Corollary is an immediate consequence of Theorems \ref{l1lowbd} and \ref{l1upbd} once we notice that if $1<p<2$ then there is a constant $C=C(n,p)>0$ such that:
$$\bigl( \mathbf{W}_{1,p}(\sigma)(x)\bigl)^{p-1} \leq C \mathbf{I}_p(\sigma)(x)
$$
for all $x\in \mathds{R}^n$.  This inequality has been proved in (\ref{wolfflessriesz}).  The opposite inequality holds if $p>2$, this is clear from (\ref{wolfflessriesz}), as the sequence space imbeddings are reversed. 

\section{Regularity away from the pole: the proof of Theorem \ref{fundreg}}\label{regularity}
In this section we will turn to considering the regularity of fundamental solutions, and in particular we will prove Theorem \ref{fundreg}.  Throughout this section we will assume the hypothesis of Theorem \ref{l1upbd} hold, and that the fundamental solution $u$ constructed there is not identically infinite.  It therefore follows from Theorem \ref{l1lowbd} that:
\begin{equation}\label{localrieszest}
\int_0^1 \frac{\sigma(B(x_0,r))}{r^{n-p}}\frac{dr}{r} = B<\infty.
\end{equation}  
By hypothesis, the constant $C(\sigma)$, defined by:
\begin{equation}\label{Csigmadef}C(\sigma) = \sup_{E\text{ compact}}\frac{\sigma(E)}{\text{cap}_{1,p}(E)},
\end{equation}
is finite, this is nothing more than a restatement of the condition (\ref{capintro}).  Thus we will assume that $C(\sigma) <C_0$, for a constant $C_0 = C_0(n,p)>0$.
The first step will be to perform some auxiliary calculations for the function $v(x)$, defined by:
\begin{equation}\label{vdefnreg}v(x) = B(n,p)\abs{x-x_0}^{\frac{p-n}{p-1}}\exp\Bigl( c\mathbf{W}_{1,p}^{\abs{x-x_0}}(\sigma)(x)+c \mathbf{I}_{p}^{\abs{x-x_0}}(\sigma) (x_0)\Bigl),
\end{equation}
for a positive constant $B(n,p)>0$ to be chosen later.  In particular, we will need to show that $v\in L^p_{\text{loc}}(\mathds{R}^n\backslash\{x_0\})$.  We will see that this is true assuming only that:
\begin{equation}\label{sobballestreg}\sigma(B(x,r))\lesssim C(\sigma)r^{n-p}\text{ for all balls }B(x,r)\subset\mathds{R}^n,\end{equation}
with the implicit constant depending on $n$ and $p$.  Display (\ref{sobballestreg}) is a special case of (\ref{Csigmadef}), using (\ref{potest}).

\begin{lem}\label{lebexpest}  There exists a constant so that if $\sigma(B(x,r))\leq C_1 r^{n-p}$ for all balls $B(x,r)\subset\mathds{R}^n$.   Then for any ball $B(x,r)\subset\mathds{R}^n$, it follows:
\begin{equation}\label{lebexpestst}\int_{B(x,r)} e^{a\mathbf{W}_{1,p}(\chi_{B(x,r)}d\sigma)} dx\leq C(r, p, C_1),
\end{equation}
for a constant $a\leq A/(C_1)^{1/(p-1)}$ with $A>0$ depending on $n$ and $p$.
\end{lem}

There are several ways one can prove this lemma, for instance one can adopt the proof of Lemma \ref{mWolffest}, leading to Corollary \ref{corexpest}, which requires some lengthy estimates of sums of dyadic cubes.  We shall avoid this by instead offering a more elegant proof, employing a regularity result from \cite{Min07}.

\begin{proof}  Fix a ball $B(x,r)$.  Then under the present assumption on $\sigma$, we may apply  Theorem 1.12 of \cite{Min07}, to find a $p$-superharmonic solution $w$ of:
\begin{equation}\begin{cases}
-\Delta_p u = \sigma \text{ in } B(x, 10r),\\
\, u = 0 \text{ on }\partial B(x,10r).
\end{cases}
\end{equation}
so that $w\in BMO(B(x,5r))$, and furthermore:
$$\sup_{B(z,s)\subset B(x,5r)} \dashint_{B(z, s)} \Bigl | w(y)  -  \dashint_{B(z,s)} w (y) dy \Bigl | dy \lesssim C_1^{1/(p-1)}.
$$
Therefore, by the John Nirenberg lemma, it follows that there exists a constant $c\lesssim C_1^{-1/(p-1)}$ so that:
\begin{equation}\label{expestafterjn}\dashint_{B(x, r)} e^{c w(y)} dy \leq \exp\Bigl(c \dashint_{B(x, r)} w(y) dy\Bigl)<\infty
\end{equation}
Employing the local Wolff potential estimate, Theorem 3.1 in \cite{KM}, it follows, for $y\in B(x,r)$ that:
\begin{equation}\begin{split}\label{lowbdwest}w(y) & \geq C \int_0^{4r}\Bigl(\frac{\sigma(B(y,s))}{s^{n-p}}\Bigl)^{1/(p-1)}\frac{ds}{s} \\
& \geq C \mathbf{W}_{1,p}(\chi_{B(x,r)} d\sigma)(y).
\end{split}\end{equation}
Substituting (\ref{lowbdwest}) into (\ref{expestafterjn}), the lemma follows.
\end{proof}

With this lemma proved, we may now prove that $v\in L^p_{\text{loc}}(\mathds{R}^n)\backslash\{x_0\}$.
\begin{lem} \label{lplocest} There exists $C_0$ so that if $C(\sigma)<C_0$, then:
$$v\in L^p_{\text{loc}}(\mathds{R}^n)\backslash\{x_0\}.$$
\end{lem}

\begin{proof}  Let $K\subset \mathds{R}^n\backslash\{x_0\}$ be a compact set, and let $B(x_j, r_j)$ be a finite cover of $K$.  Then, note that by crude estimates:
\begin{equation}\begin{split}\label{firstestvp}
\int_K & v^p dx  \lesssim d(K, x_0)^{p(n-p)/(p-1)} \exp\Bigl(cp\int_0^{|x_0|+ \text{diam}(K)}\frac{\sigma(B(x_0,r)}{r^{n-p}}\frac{dr}{r}\Bigl)\\
& \cdot \sum_j \int_{B(x_j, r_j)} \exp\Bigl(pc \int_{0}^{x-x_0}\Bigl(\frac{\sigma(B(z,r)\backslash B(x_j, 2r_j))}{r^{n-p}}\Bigl)^{1/(p-1)}\frac{dr}{r}\Bigl) \\
& \;\;\;\;\;\;\;\;\;\cdot e^{pc \mathbf{W}_{1,p}(\chi_{B(x_j,2 r_j)} d\sigma)}dx.
\end{split}\end{equation}
Employing the estimate (\ref{sobballestreg}), and recalling the definition of the constant $B$ from (\ref{localrieszest}), we readily derive:
$$\int_0^{|x_0|+ \text{diam}(K)}\frac{\sigma(B(x_0,r)}{r^{n-p}}\frac{dr}{r}\lesssim B + C(\sigma)(\log( |x_0| + \text{diam(K)})),
$$
and using the same estimate on $\sigma$, we similarly see for all $z\in B(x_j, r_j)$, that:
\begin{equation}\begin{split}\nonumber
\int_{0}^{x-x_0}\Bigl(&\frac{\sigma(B(z,r)\backslash B(x_j, 2r_j))}{r^{n-p}}\Bigl)^{1/(p-1)}\frac{dr}{r}\\
& \leq \int_{r_j}^{\text{diam}(K) + |x_0|}\Bigl(\frac{\sigma(B(x_j,r))}{r^{n-p}}\Bigl)^{1/(p-1)}\frac{dr}{r}\\
& \lesssim C(\sigma)^{1/(p-1)} \log\Bigl(\frac{\text{diam}(K)+|x_0|}{r_j}\Bigl).
\end{split}\end{equation}
Substituting these two displays into (\ref{firstestvp}), it follows:
\begin{equation}\label{secondestvp}\int_K v^p dx \leq \sum_j C(n,p,C(\sigma),r_j,K)\int_{B(x_j, r_j)} e^{pc \mathbf{W}_{1,p}(\chi_{B(x_j,2 r_j)} d\sigma)}dx.
\end{equation}
Note that under the current assumptions, we may choose $C_1\lesssim C(\sigma)$, with $C_1$ as in Lemma \ref{lebexpest}.  This is just a restatement of (\ref{sobballestreg}).  It follows that if $C_0$ is chosen small enough in terms on $n$ and $p$, then (\ref{lebexpestst}) will be valid, and therefore:
$$\int_{B(x_j, 2r_j)} e^{pc \mathbf{W}_{1,p}(\chi_{B(x_j,2 r_j)} d\sigma)}dx<\infty, \text{ for each } j.
$$
This completes the proof of the lemma.
\end{proof}

Note that in a similar way, using Corollary \ref{corexpest} instead of Lemma \ref{lebexpest}, we deduce the following lemma:
\begin{lem} \label{lpsigmaest} There exists $C_0$ so that if $C(\sigma)<C_0$, then:
$$v\in L^p_{\text{loc}}(\mathds{R}^n\backslash \{x_0\}, d\sigma).
$$
\end{lem}

We are now in a position to prove Theorem \ref{fundreg}.

\begin{proof}[Proof of Theorem \ref{fundreg}.]   Let us assume that $C_0$ has been chosen so that Lemmas \ref{lplocest} and \ref{lpsigmaest} are both valid.  To prove the theorem, we will aim to construct the sequence $\{u_m\}_m$ as in (\ref{pfundapprox}) from the proof of Theorem \ref{l1upbd} with the additional property that $u_m\in W^{1,p}_{\text{loc}}(\mathds{R}^n\backslash \{x_0\})$, with constants independent on $m$.  We will do this inductively, as in the proof of Theorem \ref{l1upbd}.  Let $u_0 = G(\,\cdot, x_0)$, with $G(x,x_0)$ as in (\ref{unperturbed}).  Note $G(\cdot, x_0)\in C_{\text{loc}}^{\infty}(\mathds{R}^n\backslash\{x_0\})$.  Suppose that we have constructed $u_1, \dots, u_{m-1}$ so that:
$$-\Delta_p u_j = \epsilon\sigma u_{j-1}^{p-1} + \delta_{x_0}, 
$$
with $u_j\leq v$, and $u_{j-1}\in W^{1,p}_{\text{loc}}(\mathds{R}^n\backslash \{x_0\})$.  Let $K$ be a compact subset of $\mathds{R}^n\backslash\{0\}$, then we claim that $u_{m-1}^{p-1}d\sigma \in W^{-1, p'}(K)$.  This will follow from the capacity strong type inequality.  Indeed, since $\sigma$ satisfies (\ref{capintro}) with constant $C(\sigma)<C_0$, it follows \cite{MazSob}, that:
$$\int |h|^p d\sigma \leq C(\sigma)\Bigl(\frac{p}{p-1}\Bigl)^p \int |\nabla h|^p dx, \text{ for all }h\in C^{\infty}_0(\mathds{R}^n),
$$
and this can be extended by continuity to functions $h \in W^{1,p}_{0}(\mathds{R}^n)$.  Now, let $h\in C_0^{\infty}(K)$, and $K'$ be a subset $K\subset\subset K' \subset\subset \mathds{R}^n\backslash \{x_0\}$  along with a function $g\in C^{\infty}_0(K')$, $g\equiv 1 \text{ on } K$, $g\geq0$. Then:
\begin{equation}\begin{split}\nonumber\int h u^{p-1}_{m-1}d\sigma & = \int hu_{m-1}^{p-1} g^{p-1}d\sigma \leq \Bigl(\int |h|^p d\sigma \Bigl)^{1/p} \Bigl(\int u_{m-1}^p g^p d\sigma\Bigl)^{\frac{p-1}{p}} \\
& \lesssim ||\nabla h||_p ||\nabla (u_{m-1}g)||^{p-1}_p\leq C_K||\nabla h||_p,
\end{split}\end{equation}
and hence $u_{m-1}^{p-1}d\sigma \in W^{-1, p'}(K)$, as claimed.  Now let $\nu_j$ be the measure:
$$\nu_j = \frac{\chi_{B(x_0, 2^{-j})}}{|B(x_0, 2^{-j})|},
$$
from Poincar\'{e}'s inequality it follows that $\nu_j \in W^{-1,p'}(B(x_0, 2^j))$.  Note in addition that $\nu_j\rightarrow \delta_{x_0}$ weakly as measures.  Invoking the theory of monotone operators, see e.g. \cite{Li69}, we assert the existence of a unique solution $w_m^j\in W^{1,p}_0(B(x_0, 2^j))$ of:
\begin{equation}\begin{cases}
-\Delta_p w_m^j = \epsilon \sigma u_{m-1}^{p-1}\chi_{B(x_0, 2^j)\backslash B(x_0, 2^{-j})} +\nu_j \text{ in } B(x_0, 2^j),\\
w_m^j \in W^{1,p}_0(B(x_0, 2^j)).
\end{cases}
\end{equation}
Furthermore, by the global potential estimate for renormalized solutions, Theorem 2.1 of \cite{PV}, it follows:
$$w_m^j(x) \leq K\epsilon \mathbf{W}_{1,p} (u_{m-1}^{p-1}d\sigma )(x) +K \mathbf{W}_{1,p}(\nu_k)(x),
$$
where the constant $K>0$ can be assumed to be the same as the constant appearing in Theorem \ref{kmpotest}.
But, for $x\not\in B(x_0, 2\cdot 2^{-j})$, a simple computation yields:
\begin{equation}\mathbf{W}_{1,p}(\nu_k)(x)\leq \frac{n-p}{p-1}2^{\frac{n-p}{p-1}}|x-x_0|^{\frac{p-n}{p-1}}.
\end{equation}
Using the hypothesis $u_{m-1}\leq v$, it follows for $x\in B(x_0, 2^j)\backslash B(x_0, 2^{1-j})$ that:
$$w_m^j(x) \leq K\epsilon \mathbf{W}_{1,p} (v^{p-1}d\sigma )(x) +K \frac{n-p}{p-1}2^{\frac{n-p}{p-1}}|x-x_0|^{\frac{p-n}{p-1}}.
$$
Let us now choose the constant $B(n,p)$ appearing in (\ref{vdefnreg}) as $B(n,p) = 2K(n-p)/(p-1)2^{\frac{n-p}{p-1}}$.  Then, by construction of $v$, it follows as in the argument around display (\ref{quasintsuper}), that we can choose $\epsilon>0$ and $C_0>0$ so that if $C(\sigma)<C_0$, then:
$$K\epsilon \mathbf{W}_{1,p} (v^{p-1}d\sigma )(x) +K \frac{n-p}{p-1}2^{\frac{n-p}{p-1}}|x-x_0|^{\frac{p-n}{p-1}}\leq v(x),
$$
and hence,
\begin{equation}\label{umjbd}w_m^j (x) \leq v(x), \text{ for all } x\in B(x_0, 2^j)\backslash B(x_0, 2 \cdot 2^{-j}).
\end{equation}
We are now in a position to derive the uniform gradient estimate.  Let $\phi \in C^{\infty}_0(B(x_0, 2^j)\backslash B(x_0, 2\cdot 2^{-j})$.  Then test the weak formulation of $w_m^j$ with the valid test function function $\phi^p \cdot w_m^j \in W^{1,p}_0(B(x_0, 2^j))$.  It follows:
$$\int |\nabla w_m^j|^p \phi^p dx = p\int |\nabla w_m^j|^{p-2}\nabla w_m^j \cdot \nabla \phi w_m^j \phi^{p-1} + \int \phi^p u_m^j u_{m-1}^{p-1}d\sigma
$$
Using Young's inequality in the first term, and utilizing the bounds (\ref{umjbd}) and $u_{m-1}\leq v$, we find that:
$$\frac{1}{p}\int|\nabla w_m^j|^p dx\leq \int v^p \phi^p d\sigma + \frac{1}{p}\int v^p |\nabla \phi|^p dx=C(n,p,C(\sigma),\text{supp}(\phi))<\infty,
$$
where Lemmas \ref{lplocest} and \ref{lpsigmaest} have been used.  Using Theorems \ref{limsuper} and \ref{weakcont}, we let $j\rightarrow \infty$ to find a solution $u_m$ of (\ref{pfundapprox}).   Furthermore, by weak compactness in $W^{1,p}$, we deduce that $u_m \in W^{1,p}_{\text{loc}}(\mathds{R}^n\backslash\{x_0\})$ with the local bound on the gradient independent of $m$.  We now follow the rest of the proof of Theorem \ref{l1upbd} from display (\ref{pfundapprox}), using weak compactness again to deduce a fundamental solution $u\in W^{1,p}_{\text{loc}}(\mathds{R}^n\backslash\{x_0\})$, so that $u\leq v$.
\end{proof}


\begin{thebibliography}{AHBV09}

\bibitem[AH96]{AH}
D.~R. Adams and L.~I. Hedberg, \emph{Function Spaces and Potential Theory},
  Grundlehren der mathematischen Wissenschaften \textbf{314}, Springer, 1996.

\bibitem[AHBV09]{AHBV}
H.~Abdul-Hamid and M-F. Bidaut-V\'{e}ron, \emph{On the connection between two
  quasilinear elliptic problems with source terms of order 0 or 1}, Preprint  (2009).

\bibitem[AG01]{AG01}
D. H. Armitage and S. J. Gardiner, \emph{Classical Potential Theory}, Springer Monographs in Mathematics (2001)

\bibitem[BV89]{BV1}
M.-F. Bidaut-V\'{e}ron, \emph{Local and global behavior of solutions of quasilinear equations of {E}mden-{F}owler type.} 
Arch. Rational Mech. Anal. \textbf{107} (1989), no. 4, 293--324. 

\bibitem[BVBV06]{BVBL}
M-F. Bidaut-V\'{e}ron, R.~Borghol, and L.~V\'{e}ron, \emph{Boundary {H}arnack inequality
  and a priori estimates of singular solutions of quasilinear elliptic
  equations}, Calc. Var. Partial Diff. Equations \textbf{27} (2006),
  no.~2, 159--177.

\bibitem[BVP01]{BVP1}
M.-F. Bidaut-V\'{e}ron and S. Pohozaev,  \emph{Nonexistence results and estimates for some nonlinear elliptic problems.}
J. Anal. Math. \textbf{84} (2001), 1--49.

\bibitem[Bir01]{Bir01}
M.~Biroli, \emph{Schr\"{o}dinger type and relaxed {D}irichlet problems for the subelliptic $p$-{L}aplacian},
Potential Anal. \textbf{15} (2001), no. 1-2, 1--16. 



\bibitem[CNS85]{CNS1}
L.~Caffarelli, L.~Nirenberg, and J.~Spruck, \emph{The {D}irichlet problem for
  nonlinear second-order elliptic equations {III}. {F}unctions of the
  eigenvalues of the {H}essian}, Acta Math. \textbf{155} (1985), no.~3-4,
  261--301.

\bibitem[COV04]{COV}
C.~Cascante, J.~M. Ortega, and I.~E. Verbitsky, \emph{Nonlinear potentials and
  two weight trace inequalities for general dyadic and radial kernels}, Indiana
  Univ. Math. J. \textbf{53} (2004), 845--882.

\bibitem[DMG94]{DMG94}
G. Dal Maso and A. Garroni, \emph{New results on the asymptotic behavior of Dirichlet problems in perforated domains}
Math. Models Methods Appl. Sci. \textbf{4} (1994), no. 3, 373--407. 

\bibitem[DMM97]{DMM97}
G. Dal Maso and A. Malusa, \emph{Some properties of reachable solutions of nonlinear elliptic equations with measure data. }
Dedicated to Ennio De Giorgi. 
Ann. Scuola Norm. Sup. Pisa Cl. Sci. (4) \textbf{25} (1997), no. 1-2, 375--396 (1998).


\bibitem[DMMOP]{DMMOP99}
G.~Dal Maso, F.~Murat, L.~Orsina, and A.~Prignet, \emph{Renormalized solutions of elliptic equations with general measure data},
Ann. Scuola Norm. Sup. Pisa Cl. Sci. \textbf{28} (1999), no.~4, 741--808. 


\bibitem[FS71]{FS}
C.~Fefferman and E.~M. Stein, \emph{Some maximal inequalities}, Amer. J. Math. 
  \textbf{93} (1971), 107--115.


\bibitem[FM00]{FM1}
V.~Ferone and F.~Murat, \emph{Nonlinear problems having natural growth in the
  gradient: an existence result when the source terms are small}, Nonlinear
  Analysis \textbf{42} (2000), 1309--1326.

\bibitem[FV09]{FV2}
M.~Frazier and I.~E. Verbitsky, \emph{Solvability conditions for a discrete
  model of {S}chr{\"{o}}dinger's equation}, Analysis, Partial Differential Equations and Applications,
The Vladimir Maz'ya Anniversary Volume.
Operator Theory: Advances and Appl. \textbf{179},   
Birk\"{a}user, 2010.


\bibitem[FV10]{FV1}
\bysame, \emph{Global {G}reen{'}s function estimates}, 
Around the Research of Vladimir Maz'ya III, Analysis and Applications, Ed. Ari Laptev,  International Math. Series \textbf{13}, Springer,  2010, 105--152.


\bibitem[FNV10]{FNV} 
M.~Frazier, F.~Nazarov, and I.~E. Verbitsky,
\emph{Global estimates for kernels of Neumann series, Green's
functions, and the conditional gauge}, Preprint (2010). 


\bibitem[GJ82]{GJ}
J.~Garnett and P.~Jones, \emph{{BMO} from dyadic {BMO}}, Pacific J. Math. \textbf{99} (1982), no.~2, 351--371.


\bibitem[GH08]{GriHan1}
A.~Grigor'yan and W.~Hansen, \emph{Lower estimates for perturbed {G}reen
  function}, J. Anal. Math. \textbf{104} (2008), 25--58.


\bibitem[HMV99]{HMV}
K.~Hansen, V.~G. Maz'ya, and I.~E. Verbitsky, \emph{Criteria of solvability for
  multidimensional {R}iccati equations}, Ark. Mat. \textbf{37} (1999), 87--120.

\bibitem[HW83]{HW}
L.~I. Hedberg and T.~Wolff, \emph{Thin sets in nonlinear potential theory},
  Ann. Inst. Fourier (Grenoble) \textbf{33} (1983), 161--187.


\bibitem[HKM06]{HKM}
J.~Heinonen, T.~Kilpel\"{a}inen, and O.~Martio, \emph{Nonlinear Potential
  Theory of Degenerate Elliptic Equations}, Dover Publications, 2006 (unabridged republication of 1993 edition, Oxford Universiy Press).

\bibitem[JMV10]{JMV10}
B. J. Jaye, V. G. Maz'ya, and I. E. Verbitsky  \emph{Existence and regularity for some elliptic equations of Schr\"{o}dinger type}, Preprint (2010)

\bibitem[JV10]{JV}
B.~J.~Jaye and I.~E.~Verbitsky, \emph{Local and global behaviour of solutions to nonlinear equations with natural growth terms},
Preprint (2010).

\bibitem[KK78]{KK1}
J.~Kazdan and R.~Kramer, \emph{Invariant criteria for existence of second-order
  quasi-linear elliptic equations}, Comm. Pure. Appl. Math. \textbf{31} (1978),
  619--645.



\bibitem[KV86]{KicVer1}
S.~Kichenassamy and L.~V\'{e}ron, \emph{Singular solutions to the $p$-{L}aplace
  equation}, Math. Ann. \textbf{275} (1986), 599--615.

\bibitem[Kil99]{Kil1}
T.~Kilpel\"{a}inen, \emph{Singular solutions of $p$-{L}aplacian type equations},
  Ark. Mat. \textbf{37} (1999), no.~2, 275--289.

\bibitem[KM92]{KM}
T.~Kilpel\"{a}inen and J.~Maly, \emph{Degenerate elliptic equations with
  measure data and nonlinear potentials}, Ann. Scoula Norm. Sup. Pisa, Cl. Sci. 
  \textbf{19} (1992), 591--613.

\bibitem[KM94]{KM1}
\bysame, \emph{The {W}iener test and potential estimates for
  quasilinear elliptic equations}, Acta Math.  \textbf{172}  (1994), 137--161.


\bibitem[Lab01]{Labiso}
D.~Labutin, \emph{Isolated singularities of fully nonlinear elliptic
  equations}, J. Differential Equations \textbf{177} (2001), no.~1, 49--76.

\bibitem[Lab02]{Lab1}
\bysame, \emph{Potential estimates for a class of fully nonlinear elliptic
  equations}, Duke Math. J. \textbf{111} (2002), no.~1, 1--49.

\bibitem[LN07]{LN1}
J. Lewis and K. Nystr\"{o}m, \emph{Boundary behaviour for $p$ harmonic functions in Lipschitz and starlike Lipschitz ring domains,} 
Ann. Sci.  Ecole Norm. Sup. (4) \textbf{40} (2007), no. 5, 765--813. 

\bibitem[Li06]{Yanyaniso}
Y.~Y. Li, \emph{Conformally invariant fully nonlinear elliptic equations and
  isolated singularities}, J. Funct. Anal. \textbf{233} (2006), no.~2,
  380--425.

\bibitem[Li69]{Li69}
J.-L. Lions, \emph{Quelques m\'{e}thodes de rŽsolution des problmes aux limites non lin\'{e}aires.} Dunod; Gauthier-Villars, Paris 1969

\bibitem[LS08]{LS1}
V.~Liskevich and I.I Skrypnik, \emph{Isolated singularities of solutions to quasilinear elliptic equations}, Potential Analysis \textbf{28} (2008),  1--16.

\bibitem[LSW63]{LSW63}
W.~Littman; G.~Stampacchia and H. F. Weinberger, \emph{Regular points for elliptic equations with discontinuous coefficients. }
Ann. Scuola Norm. Sup. Pisa (3) \textbf{17} 1963 43--77. 

\bibitem[MZ97]{MZ1}
J.~Maly and W.~Ziemer, \emph{Fine Regularity of Solutions of Elliptic Partial
  Differential Equations}, Mathematical Surveys and Monographs \textbf{51},
  Amer. Math. Soc., 1997.

\bibitem[Maz70]{MazBd}
V.~Maz'ya, \emph{The continuity at a boundary point of the solutions of quasi-linear elliptic equations,}
Vestnik Leningrad. Univ. \textbf{25} 1970 no. 13, 42--55.

\bibitem[Maz85]{MazSob}
\bysame, \emph{Sobolev Spaces}, Springer Series in Soviet Mathematics,
  Springer-Verlag, 1985 (new edition in press).

\bibitem[Min07]{Min07}
G.~Mingione, \emph{The Calder\'{o}n-Zygmund theory for elliptic problems with measure data}.
Ann. Sc. Norm. Super. Pisa Cl. Sci. (5) \textbf{6} (2007), no. 2, 195--261. 

\bibitem[MP02]{MP1}
F.~Murat and A.~Porretta, \emph{Stability properties, existence, and nonexistence of renormalized solutions for 
elliptic equations with measure data}, Comm. P.D.E. \textbf{27} (2002), no.~11-12, 2267--2310. 

\bibitem[Mur86]{Mur86}
M.~Murata, \emph{Structure of positive solutions to $(-\Delta+V)u=0$ in $R^n$}, 
Duke Math. J. \textbf{53} (1986), no.~4, 869--943. 

\bibitem[NTV99]{NTV99}
F.~Nazarov,  S.~Treil and A.~Volberg, \emph{The {B}ellman functions and two-weight inequalities for {H}aar multipliers.}
J. Amer. Math. Soc. \textbf{12} (1999), no. 4, 909--928.

\bibitem[NSS03]{NSS}
F.~Nicolosi, I.~V. Skrypnik, and I.~I. Skrypnik, \emph{Precise point-wise
  growth conditions for removable isolated singularities}, Comm.
  P.D.E. \textbf{28} (2003), 677--696.

\bibitem[PV08]{PV}
N.~C. Phuc and I.~E. Verbitsky, \emph{Quasilinear and {H}essian equations of
  {L}ane-{E}mden type}, Ann. Math. \textbf{168} (2008), 859 -- 914.


\bibitem[PV09]{PV2}
\bysame, \emph{Singular quasilinear and Hessian equations and inequalities},
  J. Funct. Anal. \textbf{256} (2009), no.~6, 1875--1905.

\bibitem[Pin07]{Pin1}
Y.~Pinchover, \emph{Topics in the theory of positive solutions of second-order elliptic and parabolic partial differential equations}, Spectral Theory and Mathematical Physics: a Festschrift in honor of Barry Simon's 60th birthday, 329--355, Proc. Sympos. Pure Math., \textbf{76}, A.M.S., 2007. 

\bibitem[PT07]{PT1}
Y.~Pinchover and K.~Tintarev, \emph{Ground state alternative for $p$-Laplacian with potential term},
Calc. Var. Partial Diff. Eqns. \textbf{28} (2007), no.~2, 179--201.


\bibitem[Roy62]{Roy62}
H. L. Royden, \emph{The growth of a fundamental solution of an elliptic divergence structure equation.} (1962) Studies in mathematical analysis and related topics pp. 333--340 Stanford Univ. Press, Stanford, Calif. 

\bibitem[Ser64]{Ser2}
J.~Serrin, \emph{Local behavior of solutions to quasi-linear equations}, Acta
  Math. \textbf{111} (1964), 247--301.

\bibitem[Ser65]{Ser1}
\bysame, \emph{Isolated singularities of solutions to quasi-linear equations},
  Acta Math. \textbf{113} (1965), 219--240.

\bibitem[SZ02]{SZ1}
J.~Serrin and H.~Zou, \emph{{C}auchy-{L}iouville and universal boundedness
  theorems for quasilinear elliptic equations and inequalities.}, Acta Math.
  \textbf{189} (2002), no.~1, 79--142.

\bibitem[TW99]{TW2}
N.~Trudinger and X.-J. Wang, \emph{Hessian measures {II}}, Ann. Math.
  \textbf{150} (1999), no.~2, 579--604.

\bibitem[TW02a]{TWHess3}
\bysame, \emph{Hessian measures {III}}, J. Funct. Anal. \textbf{193} (2002),
  no.~1, 1--23.

\bibitem[TW02b]{TW1}
\bysame, \emph{On the weak continuity of elliptic operators and applications to
  potential theory}, \textbf{124} Amer. J. Math. (2002),  369--410.

\bibitem[TW09]{TWQuas}
\bysame, \emph{Quasilinear elliptic equations with signed measure data},
Discrete Contin. Dyn. Syst. \textbf{23} (2009), no. 1-2, 477--494. 

\bibitem[V10]{verb}
I.~E.~Verbitsky, \emph{Green's function estimates for some linear and nonlinear problems},
Proc. Indam School on symmetry for elliptic PDEs: 30 years after a conjecture of De Giorgi and related problems, Rome, Italy, May 25--29, 2009, Contemp. Math., Amer. Math. Soc., to appear.  

\bibitem[Ver96]{veron}
L.~V\'{e}ron, \emph{Singularities of Solutions of Second-Order Quasilinear
  Equations}, Chapman and Hall, 1996.

\bibitem[Wan09]{Wang1}
X.-J. Wang, \emph{The $k$-{H}essian Equation}, Lecture Notes Math. \textbf{1977}, Springer, 2009.

\end{thebibliography}
\end{document}